\documentclass{conm-p-l}
\title{Some problems and conjectures about $L^2$-invariants}
\author{Dominik Kirstein}
\email{kirstein@math.uni-bonn.de}
\urladdr{https://dkirstein.github.io}
 \author{Christian Kremer}
\email{kremer@mpim-bonn.mpg.de}
\urladdr{https://sites.google.com/view/christian-kremer-math}
\author{Wolfgang L\"uck}
\email{wolfgang.lueck@him.uni-bonn.de}
          \urladdr{http://www.him.uni-bonn.de/lueck}
\address{Mathematical Institute of the University of Bonn\\
                Endenicher Allee 60\\
                53115 Bonn, Germany}
     \makeatletter
     \@namedef{subjclassname@2020}{%
  \textup{2020} Mathematics Subject Classification}
\makeatother

\usepackage[colorlinks = true,
            linktocpage = true,
            linkcolor = magenta,
            citecolor = cyan]{hyperref}
\usepackage{color}
\usepackage{pdfsync}
\usepackage{calc}
\usepackage{mathtools}
\usepackage{enumerate,amssymb}
\usepackage[arrow,curve,matrix,tips,2cell]{xy} \usepackage{tikz-cd}
\SelectTips{eu}{10} \UseTips \UseAllTwocells

\DeclareMathAlphabet{\matheurm}{U}{eur}{m}{n} \DeclareMathAlphabet\EuR{U}{eur}{m}{n}
\SetMathAlphabet\EuR{bold}{U}{eur}{b}{n}

\makeindex 

\newtheorem{theorem}{Theorem}[section] 
\newtheorem{proposition}[theorem]{Proposition} 
\newtheorem{conjecture}[theorem]{Conjecture} \newtheorem{problem}[theorem]{Problem}
\newtheorem{question}[theorem]{Question}

\theoremstyle{definition} 
\newtheorem{definition}[theorem]{Definition} \newtheorem{example}[theorem]{Example}

\theoremstyle{remark} \newtheorem{remark}[theorem]{Remark}

\numberwithin{equation}{section}

\newcommand{\comsquare}[8] 
{\begin{CD}
    #1 @>#2>> #3\\
    @V{#4}VV @V{#5}VV\\
    #6 @>#7>> #8
  \end{CD}
}

\newcommand{\xycomsquare}[8] 
{\xymatrix {#1 \ar[r]^{#2} \ar[d]^{#4} &
    #3 \ar[d]^{#5}  \\
    #6\ar[r]^{#7} & #8 } }

\newcommand{\xycomsquareminus}[8] 
{\xymatrix{#1 \ar[r]^-{#2} \ar[d]^-{#4} &
    #3 \ar[d]^-{#5}  \\
    #6\ar[r]^-{#7} & #8 } }

\newcommand{\cala}{\mathcal{A}} \newcommand{\calb}{\mathcal{B}}
\newcommand{\cald}{\mathcal{D}} 

\newcommand{\calf}{\mathcal{F}} 
\newcommand{\calg}{\mathcal{G}}
\newcommand{\calr}{\mathcal{R}}

\newcommand{\calc}{{\mathcal C}} \newcommand{\caln}{{\mathcal N}}
\newcommand{\calu}{{\mathcal U}}

  \newcommand{\IC}{{\mathbb C}}
  \newcommand{\IF}{{\mathbb F}}

\newcommand{\IN}{{\mathbb N}}  
\newcommand{\IQ}{{\mathbb Q}} \newcommand{\IR}{{\mathbb R}}

\newcommand{\IZ}{{\mathbb Z}}

\newcommand{\Ab}{\cat{Ab}}

\newcommand{\Cond}[1]{\cat{Cond}(#1)}

\newcommand{\edCH}{\cat{edCH}}

\newcommand{\Or}{\cat{Or}}

\newcommand{\an}{\operatorname{an}}

 \newcommand{\colim}{\operatorname{colim}}

\newcommand{\costoper}{\operatorname{cost}} 
 \newcommand{\dom}{\operatorname{dom}}

\newcommand{\GL}{\operatorname{GL}} 
 
 \newcommand{\id}{\operatorname{id}}
 \newcommand{\im}{\operatorname{im}}

 \newcommand{\lcm}{\operatorname{lcm}}

 \newcommand{\op}{\operatorname{op}}
\newcommand{\Ore}{\operatorname{Ore}} 
\newcommand{\PL}{\operatorname{PL}} \newcommand{\RG}{\operatorname{RG}}
 
\newcommand{\rk}{\operatorname{rk}} 
\newcommand{\sign}{\operatorname{sign}}

\newcommand{\topo}{\operatorname{top}} \newcommand{\Tor}{\operatorname{Tor}}
\newcommand{\tors}{\operatorname{tors}} \newcommand{\tr}{\operatorname{tr}}
 
 \newcommand{\vol}{\operatorname{vol}}

\newcommand{\version}[1] 
{\begin{center} last edited on #1\\
    last compiled on \today\\
    name of texfile: \jobname
  \end{center}
}

\newcounter{commentcounter}

\newcommand{\cat}[1]{\mathsf{#1}}
\newcommand{\Mod}{\cat{Mod}}

\newcommand{\Set}{\cat{Set}}
\newcommand{\Top}{\cat{Top}}
\newcommand{\CW}{\cat{CW}}

\DeclareMathOperator{\Hom}{Hom}
\DeclareMathOperator{\Fun}{Fun}

\newcommand{\Solid}[1]{\cat{Solid}(#1)}

\begin{document}




                                 \typeout{---------------------------- L2approx_survey.tex
                                   ----------------------------}


                                 \typeout{------------------------------------ Abstract
                                   ----------------------------------------}

  \begin{abstract}
    In this paper we discuss open problems and conjectures concerning $L^2$-invariants.
  \end{abstract}

  \maketitle


  \typeout{------------------------------- Section 1: Introduction
    --------------------------------}

  \section{Introduction}

  In this article we give a survey on open problems and conjectures concerning
  $L^2$-invariants.  We cover the whole portfolio and not only certain aspects as they are
  considered in the previous more specialized (and within their scope more detailed) survey
  articles~\cite{Lueck(2016_l2approx), Lueck(2021survey)}.  Moreover, we include some new
  results and problems, which have occurred after these two survey articles were written.
  The reader may select a specific topic by looking at the table of contents below.

  \subsection*{Acknowledgements}
  This paper is funded by the ERC Advanced Grant ``KL2MG-interactions'' (no.  662400) of the
  third author granted by the European Research Council and by the Deutsche
  Forschungsgemeinschaft  under Germany's Excellence Strategy \-- GZ 2047/1, Projekt-ID 390685813,
  Hausdorff Center for Mathematics at Bonn.
  All three authors are supported by the Max Planck Institute for Mathematics in Bonn.

  We thank Grigori Avramidi for comments on an earlier version of this article and Emma Brink for helpful discussions concerning Section~\ref{sec:L2-invariants_and_condensed_mathematics}.

  The paper is organized as follows:
  \tableofcontents


  \typeout{------------------------ Section 2: Basics about $L^2$-invariants  --------------------}

  \section{Basics about $L^2$-invariants}\label{sec:Basics_about_L2-invariants}

  We present some basic definitions, constructions, and properties of $L^2$-invariants for the reader's convenience.
  The technical aspects of this section will  not be needed in the other sections.


  \subsection{Group von Neumann algebras and their dimension function}%
  \label{subsec:Group_von_Neumnan_algebras_and_their_dimension_functions}

  The \emph{group von Neumann algebra}
  \begin{equation}
    \caln(G)  := \calb(L^2(G))^G
    \label{caln(G)}
  \end{equation}
  of the (discrete) group $G$ is defined to be  the $\IC$-algebra of $G$-equivariant bounded operators
  from $L^2(G)$ to $L^2(G)$, where $L^2(G)$ is the complex Hilbert space of square integrable functions
  from $G$ to $\IC$. It can also be viewed as a
  completion of the complex group ring $\IC G$ with respect to the weak (or equivalently strong) operator topology.
    We will view it in the sequel just as a ring with unit and
  essentially forget all the functional analytic aspects. It has much nicer properties than $\IC G$, for example because it is \emph{semi-hereditary}, i.e., any finitely generated
  submodule of a projective $\caln(G)$-module is projective again. More information about
  von Neumann algebras and their basic properties can be found for instance
  in~\cite[Section~2.2]{Kammeyer(2019)} and~\cite[Section~1.1 and Section~9.1]{Lueck(2002)}.

  The main feature is the \emph{dimension function} that assigns to every $\caln(G)$-module
  $M$ an element
  \begin{equation}
    \dim_{\caln(G)}(M) \in \IR^{\ge 0} \amalg \{\infty\}.
    \label{dim(M)}
  \end{equation}
  for $\IR^{\ge 0} = \{r \in \IR \mid r \ge 0\}$.  We summarize its main properties:

  \begin{enumerate}
  	\item It is \textit{additive} on short exact sequences of $\caln(G)$-modules, also in the extreme cases following the convention that for $x \in \IR^{\ge 0} \amalg \{\infty\}$ we set $x + \infty = \infty$.
  	\item The dimension function extends the classical Murray-von Neumann dimension for
    finitely generated projective $\caln(G)$-modules defined by the Hattori-Stalling rank
    associated to the standard trace of $\caln(G)$.
    In particular, we have $\dim_{\caln(G)}(\caln(G)) = 1$, which together with additivity implies that the dimension of a finitely generated $\caln(G)$-module is finite.
  	\item It is \textit{cofinal} in the sense that for any directed union $M = \bigcup_{i \in I } M_i$ of $\caln(G)$-modules the equality $\dim_{\caln(G)}(M) = \sup\{\dim_{\caln(G)}(M_i) \mid i \in I\}$ holds.
    \item It satisfies the \textit{continuity property}. A
    special case says that for any finitely generated $\caln(G)$-module $M$ one can find a finitely generated projective $\caln(G)$-module $PM$ together with two exact sequences $0 \to TM \to M \to PM \to 0$ and $0 \to \caln(G)^n \to \caln(G)^n \to TM \to 0$ for which
    $\dim_{\caln(G)}(TM) = 0$ and $\dim_{\caln(G)}(M) = \dim_{\caln(G)}(PM)$ hold. The general version can be found in~\cite[Theorem 6.7]{Lueck(2002)}.
  	\item We have
  	$\dim_{\caln(G)}(M) = 0$ if and only if every projective submodule of $M$ is trivial.
    Moreover, a finitely generated projective $\caln(G)$-module $M$ is trivial if and only if
  	$\dim_{\caln(G)}(M) = 0$ holds.
  \end{enumerate}
  It turns out that the dimension function is uniquely determined by the properties (1) - (4).
  In the special case where $G$ is finite and $M$ is an $\caln(G)$-module one has $\caln(G) = \IC G$, and $|G| \cdot \dim_{\caln(G)}(M)$
  is the (classical) dimension of the underlying complex vector space of $M$. If $G$
  contains an element of infinite order, then every element $r \in \IR^{\ge 0}$ can be
  realized as $r = \dim_{\caln(G)}(M)$ for some finitely generated projective $\caln(G)$-module
  $M$. More information about the dimension function can be found for instance in~\cite[Sections~2.3
  and
  4.2]{Kammeyer(2019)},~\cite{Lueck(1998a)},~\cite{Lueck(1998b)},~\cite[Chapter~6]{Lueck(2002)}.


  \subsection{$L^2$-Betti numbers}\label{subsec:L2-Betti_numbers}

  Let $Y$ be a $G$-space. Its singular chain complex $C^s_*(Y)$ is acted upon by $G$, so it acquires the structure of a chain complex of
  $\IZ G$-modules.  Note that $\caln(G)$ can be viewed as an $\caln(G)$-$\IZ
  G$-bimodule.  Then we obtain an $\caln(G)$-chain complex $\caln(G) \otimes_{\IZ G}
  C^s_*(X)$. Its $n$th homology $H_n(\caln(G) \otimes_{\IZ G} C^s_*(X))$ is a
  $\caln(G)$-module. We define the \emph{$n$th-$L^2$-Betti number} of the $G$-space $Y$ to be
  \begin{equation}
    b_n^{(2)}(Y;\caln(G)) := \dim_{\caln(G)}\bigl(H_n(\caln(G) \otimes_{\IZ G} C^s_*(Y))\bigr)
    \in \IR^{\ge 0} \amalg \{\infty\}.
    \label{b_n_upper_(2)(X,caln(G))}
  \end{equation}
  If $Y$ is a $G$-$CW$-complex, we can replace $C^s_*(Y)$ by the cellular $\IZ G$-chain
  complex $C_*(Y)$ in~\eqref{b_n_upper_(2)(X,caln(G))}. If $Y$ is a $G$-$CW$-complex such
  that $Y/G$ is a finite $d$-dimensional  $CW$-complex,  then
  $b_n^{(2)}(Y;\caln(G))$ is finite for every $n \ge 0$ and vanishes for $n > d$.

  We will mainly be interested in the case where $X$ is a path connected space possessing a universal covering $p \colon \widetilde{X} \to X$. Recall that
  $\widetilde{X}$ then comes with a free $\pi$-action for $\pi$ the fundamental group of $X$, and
  we abbreviate
  \begin{equation}
    b_n^{(2)}(\widetilde{X})  = b_n^{(2)}(\widetilde{X};\caln(\pi)) \in \IR^{\ge 0} \amalg \{\infty\}.
    \label{b_n_upper_(2)(widetilde(G)}
  \end{equation}
  If $X$ is not path connected, one defines $b_n^{(2)}(\widetilde{X})$ to be
  $\sum_{C \in \pi_0(X)} b_n^{(2)}(\widetilde{C})$.  If $X$ is a $CW$-complex of finite
  type, then $b_n^{(2)}(\widetilde{X})$ is finite.  If $M$ is a closed Riemannian manifold,
  then $b_n^{(2)}(\widetilde{M})$ agrees with the original definition of the $L^2$-Betti
  number due to Atiyah~\cite{Atiyah(1976)} in terms of heat kernel on the universal
  covering, namely, we have
  \begin{eqnarray}
  b_n^{(2)}(\widetilde{M})
  & = &
  \lim_{t \to \infty} \int_{\calf} \tr(e^{-t\widetilde{\Delta}_n}(\widetilde{x},\widetilde{x})) \operatorname{dvol}
  \label{L_upper_2-Betti_number_and_heat_kernel}
  \end{eqnarray}
  where $\calf \subseteq \widetilde{X}$ is a fundamental domain
  for the $\pi$-action on $\widetilde{M}$ and $\tr(e^{-t\widetilde{\Delta}_n}(\widetilde{x},\widetilde{x}))$
  is the trace of the heat kernel on $\widetilde{M}$ at $(\widetilde{x}, \widetilde{x})$
  for $\widetilde{x} \in \widetilde{X}$ at the time $t$ in degree $n$.

  If $G$ is a (discrete) group, we define its \emph{$n$th-$L^2$-Betti number}
  \begin{equation}
    b_n^{(2)}(G) := b_n^{(2)}(EG;\caln(G)) \in \IR^{\ge 0} \amalg \{\infty\}.
    \label{b_n_upper_(2)(G)}
  \end{equation}
  where $EG$ is a the total space of the universal principal $G$-bundle $EG \to BG$.  If
  $\underline{E}G$ is the classifying space for proper $G$-actions, see for
  instance~\cite{Lueck(2005s)}, then we conclude
  $b_n^{(2)}(G) := b_n^{(2)}(\underline{E}G;\caln(G))$ from~\cite[Theorem~6.54~(2) on
  page~265]{Lueck(2002)}. This is useful since there are often small models of
  $\underline{E}G$ available when $EG$ may have no such nice models.
  For example, when $G$ has torsion every $CW$-model for $EG$ has cells in infinitely many dimension, whereas one might construct a finite dimensional model for $\underline{E}G$ to conclude vanishing of $L^2$-Betti numbers in high degrees.

  Roughly speaking, the usefulness of the $L^2$-Betti numbers stems often from the fact that they vanish quite often and pose interesting obstructions to the solution of geometric problems if they don't.
  For instance, for a connected $CW$-complex $X$ we have
  $b_0(X) = 1$, whereas $b_0^{(2)}(\widetilde{X})$ vanishes if $\pi$ is infinite and is equal to
  $|\pi|^{-1}$ if $\pi$ is finite. We mention as an appetizer two rather surprising
  examples.

  \begin{example}
  	If the group $G$ contains an amenable  infinite  normal subgroup $H \subseteq G$,
  	then $b_n^{(2)}(G)$ vanishes for every $n \ge 0$.  This implies that the Euler
  	characteristic $\chi(BG)$ vanishes, if we additionally assume that $BG$ has a finite
  	$CW$-model.  Note that we do not require anything about $BH$. This is a special case of~\cite[Theorem 7.2]{Lueck(2002)}.
  \end{example}

  \begin{example}
  	If $1 \to K \to G \to Q \to 1$ is an extension of infinite groups such that $K$ is finitely
  	generated and $G$ is finitely presented, then $b_1^{(2)}(G)$ vanishes and hence the
  	deficiency of $G$ is bounded from above by $1$.  Recall that the deficiency is the maximum
  	of the set of integers $g(P) - r(P)$, where $P$ runs through all finite presentations $P$
  	of $G$ and $g(P)$ is the number of generators and $r(P)$ the number of relations.~\cite[Theorem 6.1]{Lueck(1994b)}
  \end{example}


  \subsection{Fuglede-Kadison determinants}\label{subsec:Fuglede-Kadison_determinants}

  Given an $\caln(G)$-map  $f \colon P \to Q$  of finitely generated projective $\caln(G)$-modules,
  one can define its \emph{Fuglede-Kadison determinant}
  \begin{equation}
    {\det}_{\caln(G)}(f) \in \IR^{\ge 0}  := \{r \in \IR \mid r \ge 0\}.
    \label{det(f)}
  \end{equation}
  Let us briefly explain its definition, even though it will be of no concern to us later.
  One can interprete $f$ as a bounded $G$-equivariant operator $f \colon V \to W$ of
  \emph{finitely generated Hilbert $\caln(G)$-modules}, i.e., $V$ is a Hilbert space $V$
  together with a linear isometric $G$-action such that there exists an isometric linear
  $G$-embedding of $V$ into $L^2(G)^r$ for some natural number $r$, and analogously for $W$.
  Denote by $\{E_{\lambda}^{f^*f} \mid \lambda \in \IR\}$ the (right-continuous) family of
  spectral projections of the positive operator $f^*f \colon V \to V$.  Define the
  \emph{spectral density function of} $f$ by
    \[
      F_f \colon \IR \to \IR^{\ge 0} \quad \lambda \mapsto
      \dim_{\caln(G)}\bigl(\im(E_{\lambda^2}^{f^*f})\bigr).
    \]
    The spectral density function is monotone, non-decreasing, right-continuous and satisfies $F(0) = \dim_{\caln(G)}(\ker(f))$.
    Let $dF$ be the unique measure on the Borel $\sigma$-algebra
    on $\IR$ that satisfies $dF((a,b]) = F(b)-F(a)$ for $a < b$.  Define
    Fuglede-Kadison determinant $\det_{\caln(G )}(f) \in \IR^{\ge 0}$   to be
    \[
    {\det}_{\caln(G )}(f) = \exp\left(\int_{0+}^{\infty} \ln(\lambda) \; dF\right)
    \]
    if the Lebesgue integral $\int_{0+}^{\infty} \ln(\lambda) \; dF $ is a real
    number, and to be  $0$ if we have $\int_{0+}^{\infty} \ln(\lambda) \; dF = - \infty$.
    We say  that $f$ is of \emph{determinant class} if $\det(f) \not= 0$.
    With our conventions we have $\det(0) = 1$ for the zero map $0 \colon P \to Q$

    \begin{example}[Fuglede-Kadison determinant for finite
      groups]\label{exa:det_for_finite_groups}
      To illustrate this definition, we look at the example where $G$ is finite. Then
      $\caln(G)$ is the same as $\IC G$ and we can think of $f$ as a $\IC G$-linear map of
      finite-dimensional unitary $G$-representations. The spectral density function $F_f$ is
      the right-continuous step function, whose value at $\lambda$ is the sum of the complex
      dimensions of the eigenspaces of $f^*f$ for eigenvalues $\mu \le \lambda^2$ divided by
      the order of $G$. Let $\lambda_1$, $\lambda_2$, $\ldots$, $\lambda_r$ be the non-zero
      eigenvalues of $f^{\ast}f$ with multiplicity $\mu_i$.
      One computes
      \begin{multline*} {\det}_{\caln(G)}(f) = \exp\left(\sum_{i = 1}^r \frac{\mu_i}{|G|}
          \cdot \ln(\sqrt{\lambda_i})\right) = \prod_{i=1}^r \lambda_i^{\frac{\mu_i}{2\cdot
            |G|}} = {\det}_{\IC}\bigl(\overline{f^{\ast}f}\big)^{\frac{1}{2\cdot |G|}}.
      \end{multline*}
      Here, $\overline{f^{\ast}f}$ denotes the automorphism of the orthogonal complement of the kernel
      of $f^{\ast}f$ induced by $f^{\ast}f$ and ${\det}_{\IC}\bigl(\overline{f^{\ast}f})$ is put to be $1$ if $f$ is the zero
      operator.
      If $f \colon \IC G^m \to \IC G^m$ is an automorphism, we get
      \[ {\det}_{\caln(G)}(f) = \left|{\det}_{\IC}(f)\right|^{\frac{1}{|G|}}.
      \]
    \end{example}

    More information about the Fuglede-Kadison determinant can be found for instance
    in~\cite[Section~3.2]{Lueck(2002)}.


  \subsection{$L^2$-torsion}\label{subsec:L2-torsion}

  Let $X$ be a connected finite $CW$-complex. Let $C_*(\widetilde{X})$ be the cellular
  $\IZ \pi$-chain complex of its universal covering for $\pi$ the fundamental group of
  $X$. Then $\caln(\pi) \otimes_{\IZ \pi} C_*(\widetilde{X})$ is an $\caln(\pi)$-chain complex,
  which is finite dimensional and consists of finitely generated free $\caln(\pi)$-modules.
  If $c_n \colon C_n(\widetilde{X}) \to C_{n-1}(\widetilde{X})$ is the $n$th differential
  of $C_*(\widetilde{X})$, then
  \[
  \id_{\caln(\pi)} \otimes_{\IZ \pi} c_n \colon \caln(\pi) \otimes_{\IZ \pi} C_n(\widetilde{X})
  \to \caln(\pi)\otimes_{\IZ \pi} C_{n-1}(\widetilde{X})
\]
 is an $\caln(\pi)$-map of finitely
  generated free $\caln(\pi)$-modules and its Fuglede-Kadison determinant
  $\det_{\caln(\pi)}(\id_{\caln(\pi)} \otimes_{\IZ \pi} c_n) \in \IR^{\ge 0} $ is defined. We say
  that $X$ is \emph{of determinant class} if
  $\det_{\caln(\pi)}\bigl(\id_{\caln(\pi)} \otimes_{\IZ \pi} c_n\bigr) \not = 0$ holds for every
  $n \ge 1$, and in this case we define the \emph{$L^2$-torsion} of $X$ to be
  \begin{equation}
    \rho^{(2)}(\widetilde{X}) = - \sum_{n \ge 1}
   (-1)^n  \cdot \ln\bigl({\det}_{\caln(\pi)}(\id_{\caln(\pi)} \otimes_{\IZ \pi} c_n) \bigr) \in \IR.
  \label{rho_upper_(2)(widetilde(x))}
  \end{equation}

  We say that $X$ is \emph{$L^2$-acyclic} if $b_n^{(2)}(\widetilde{X}) = 0$ for $n \ge 0$.
  We say that $X$ is \emph{$\det$-$L^2$-acyclic} if it is $L^2$-acyclic and of determinant
  class.  If $X$ is not path connected, we define
  $\rho^{(2)}(\widetilde{X}) = \sum_{C \in \pi_0(X)} \rho^{(2)}(\widetilde{C})$.

  The definition of $L^2$-torsion in the analytic setting goes back to Lott~\cite
  {Lott(1992a)} and Mathai~\cite{Mathai(1992)}, and in the topological setting to
  L\"uck-Rothenberg~\cite{Lueck-Rothenberg(1991)}.  The equality of these two versions has
  been proved in~\cite{Burghelea-Friedlander-Kappeler-McDonald(1996a)}.


  \subsection{Basic properties of $L^2$-Betti numbers and $L^2$-torsion}%
  \label{subsec:Basic_properties_of_L2-Betti_numbers_and_L2-torsion}
  \ \\[1mm]
  Here is a list of basic properties of $L^2$-Betti numbers and $L^2$-torsion.

  \begin{theorem}\label{thm:main_properties_of_rho2(widetildeX)}\
  \begin{enumerate}
  \item\label{list:main_properties_of_rho2(widetildeX):homotopy_invariance} \emph{(Simple)
      Homotopy invariance}, see~\cite[Theorem~1.35~(1) on page~37 and Theorem~3.96~(i) on
    page~163]{Lueck(2002)},\cite[Theorem~6.7~(2)]{Lueck(2018)}.

    \begin{enumerate}
    \item If $X$ and $Y$ are homotopy equivalent spaces, then we get for $n \ge 0$
      \[b_n^{(2)}(\widetilde{X}) = b_n^{(2)}(\widetilde{Y});
      \]
    \item Suppose that $X$ and $Y$ are homotopy equivalent finite $CW$-complexes
      and $\widetilde{X}$ or $\widetilde{Y}$ is $\det$-$L^2$-acyclic.  Then
      both $\widetilde{X}$ and $\widetilde{Y}$ are $\det$-$L^2$-acyclic;

    \item Let $f \colon X \to Y$ be  a homotopy equivalence of finite $CW$-complexes.
      Assume that $\widetilde{X}$ and
      $\widetilde{Y}$ are $\det$-$L^2$-acyclic. Suppose that
      $f$ is simple, or  that $\pi_1(X)$ satisfies the Determinant
      Conjecture~\ref{con:Determinant_Conjecture},  or that the $K$-theoretic Farrell-Jones Conjecture
      for $\IZ \pi_1(X)$  holds.  Then
      \[
        \rho^{(2)}(\widetilde{Y}) = \rho^{(2)}(\widetilde{X});
      \]
    \end{enumerate}

  \item\label{list:main_properties_of_rho2(widetildeX):Euler-Poincare_formula}
    \emph{Euler-Poincar\'e formula}, see~\cite[Theorem~1.35~(2) on page~37]{Lueck(2002)}.
    \\[1mm]
    We get for the Euler characteristic $\chi(X)$ of a finite $CW$-complex $X$
    \[
      \chi(X) = \sum_{n \ge 0 } (-1)^n \cdot b_n^{(2)}(\widetilde{X});
    \]

  \item\label{list:main_properties_of_rho2(widetildeX):sum_formula}
    \emph{Sum formula}, see~\cite[Theorem~3.96~(2)  on page~164]{Lueck(2002)}.\\[1mm]
    Consider a pushout of finite $CW$-complexes such that $j_1$ is an inclusion of
    $CW$-complexes, $j_2$ is cellular, and $X$ inherits its $CW$-complex structure from
    $X_0$, $X_1$ and $X_2$
    \[
      \xymatrix{X_0 \ar[r]^-{j_1} \ar[d]_{j_2} & X_1 \ar[d]^{i_1} \\ X_2 \ar[r]_-{i_2} & X.}
    \]
    Assume that for $k=0,1,2$ the map $\pi_1(i_k,x)\colon \pi_1(X_k,x) \to \pi_1(X,j_k(x) )$
    induced by the obvious map $i_k\colon X_k \to X$ is injective for all base points $x$ in
    $X_k$.
    \begin{enumerate}
    \item If $\widetilde{X_0}$, $\widetilde{X_1}$, and $\widetilde{X_2}$ are $L^2$-acyclic,
      then $\widetilde{X}$ is $L^2$-acyclic;

    \item If $\widetilde{X_0}$, $\widetilde{X_1}$, and $\widetilde{X_2}$ are
      $\det$-$L^2$-acyclic, then $\widetilde{X}$ is $\det$-$L^2$-acyclic and we get
      \[
        \rho^{(2)}(\widetilde{X}) = \rho^{(2)}(\widetilde{X_1}) +
        \rho^{(2)}(\widetilde{X_2}) - \rho^{(2)}(\widetilde{X_0});
      \]
    \end{enumerate}

  \item\label{list:main_properties_of_rho2(widetildeX):Poincare_duality}
    \emph{Poincar\'e duality}, see~\cite[Theorem~1.35~(3) on page~37 and Theorem~3.96~(3)  on page~164]{Lueck(2002)}.\\[1mm]
    Let $M$ be a closed manifold of dimension $d$
    \begin{enumerate}
    \item Then
      \[
        b_n^{(2)}(\widetilde{M}) = b_{d-n}^{(2)}(\widetilde{M});
      \]
    \item Suppose that $n$ is even and $\widetilde{M}$ is $\det$-$L^2$-acyclic. Then
      \[
        \rho^{(2)}(\widetilde{M}) = 0;
      \]
    \end{enumerate}

  \item\label{list:main_properties_of_rho2(widetildeX):product_formula}
    \emph{Product formula}, see~\cite[Theorem~1.35~(4) on page~37 and Theorem~3.96~(4)  on page~164]{Lueck(2002)}.\\[1mm]
    Let $X$ and $Y$ be finite $CW$-complexes.

    \begin{enumerate}

    \item Then
      \[
        b_n^{(2)}(\widetilde{X \times Y}) = \sum_{\substack{i,j \in \IN\\ n = i+j}}
        b_i^{(2)}(\widetilde{X}) \cdot b_j^{(2)}(\widetilde{Y});
      \]
    \item Suppose that $\widetilde{X}$ is $\det$-$L^2$-acyclic. Then
      $\widetilde{X \times Y}$ is $\det$-$L^2$-acyclic and
      \[
        \rho^{(2)}(\widetilde{X \times Y}) = \chi(Y) \cdot \rho^{(2)}(\widetilde{X});
      \]
    \end{enumerate}
  \item\label{list:main_properties_of_rho2(widetildeX):multiplicativity}
    \emph{Multiplicativity}, see~\cite[Theorem~1.35~(9) on page~38 and Theorem~3.96~(5)  on page~164]{Lueck(2002)}.\\[1mm]
    Let $X \to Y$ be a finite covering of finite $CW$-complexes with $d$ sheets.

    \begin{enumerate}
    \item Then we get for $n \ge 0$
      \[
        b_n^{(2)}(\widetilde{X}) = d \cdot b_n^{(2)}(\widetilde{Y});
      \]
    \item Then $\widetilde{X}$ is $\det$-$L^2$-acyclic if and only if $\widetilde{Y}$ is
      $\det$-$L^2$-acyclic, and in this case
      \[
        \rho^{(2)}(\widetilde{X}) = d \cdot \rho^{(2)}(\widetilde{Y});
      \]
    \end{enumerate}

  \item\label{list:main_properties_of_rho2(widetildeX):sofic_det-class}
    \emph{Determinant class}.\\[1mm]
    If $\pi_1(C)$ satisfies the Determinant Conjecture~\ref{con:Determinant_Conjecture} for
    each component $C$ of the finite $CW$-complex $X$, then $\widetilde{X}$ is of
    determinant class;

  \item\label{list:main_properties_of_rho2(widetildeX):det-class:b_0_upper(2)}
    \emph{$0$th $L^2$-Betti number}, see~\cite[Theorem~1.35~(8) on page~38]{Lueck(2002)}.\\[1mm]
    If $X$ is a connected finite $CW$-complex with fundamental group $\pi$, then
    \[
      b_0^{(2)}(\widetilde{X}) =
      \begin{cases}
        \frac{1}{|\pi|} & \text{if}\; \pi \; \text{is finite};
        \\
        0 & \text{otherwise;}
      \end{cases}
    \]

  \item\label{list:main_properties_of_rho2(widetildeX):det-class:fibrations} \emph{Fibration
      formula}, see~\cite[Lemma~1.41 on page~45 and Corollary~3.103 on
    page~166]{Lueck(2002)}.  \smallskip
    \begin{enumerate}
    \item Let $p \colon E \to B$ a fibration whose base space $B$ is a connected finite
      $CW$-complex and whose fiber is homotopy equivalent to a finite $CW$-complex $Z$. Suppose
      that for every $b \in B$ and $x \in F_b := p^{-1}(b)$ the inclusion $p^{-1}(b) \to E$
      induces an injection on the fundamental groups $\pi_1(F_b,x) \to \pi_1(E,x)$, and that
      $Z$ is $L^2$-acyclic.

      Then $E$ is homotopy equivalent to a finite $CW$-complex $X$ which is $L^2$-acyclic;

    \item Let $F \xrightarrow{i} E \xrightarrow{p} B$ be locally trivial fiber bundle of
      finite $CW$-complexes.  Suppose that $B$ is connected, the map
      $\pi_1(F,x) \to\pi_1(E,i(x))$ is bijective for every base point $x \in F$, and
      $\widetilde{F}$ is $\det$-$L^2$-acyclic.

      Then $\widetilde{E}$ is $\det$-$L^2$-acyclic and
      \[
        \rho^{(2)}(\widetilde{E}) = \chi(B) \cdot \rho^{(2)}(\widetilde{F});
      \]
    \end{enumerate}

  \item\label{list:main_properties_of_rho2(widetildeX):S_upper_1-actions}
    \emph{$S^1$-actions}, see~\cite[Theorem~1.40 on page~43 and Theorem~3.105  on page~168]{Lueck(2002)}.\\[1mm]
    Let $X$ be a connected compact $S^1$-$CW$-complex, for instance a closed smooth manifold
    with smooth $S^1$-action. Suppose that for one orbit $S^1/H$ (and hence for all orbits)
    the inclusion into $X$ induces a map on $\pi_1$ with infinite image (so in particular the
    $S^1$-action has no fixed points). Then $\widetilde{X}$ is $\det$-$L^2$-acyclic and
    $\rho^{(2)}(\widetilde{M})$ vanishes;

  \item\label{list:main_properties_of_rho2(widetildeX):aspherical}
    \emph{Aspherical spaces}, see~\cite[Theorem~3.111  on page~171 and Theorem~3.113  on page~172]{Lueck(2002)}.
    \smallskip
    \begin{enumerate}

    \item Let $M$ be an aspherical closed smooth manifold with a smooth $S^1$-action.  Then
      the conditions appearing in
      assertion~\eqref{list:main_properties_of_rho2(widetildeX):S_upper_1-actions} are
      satisfied and hence $\widetilde{M}$ is $\det$-$L^2$-acyclic and
      $\rho^{(2)}(\widetilde{X})$ vanishes;

    \item If $X$ is an aspherical finite $CW$-complex whose fundamental group contains an
       elementary amenable infinite normal subgroup, then $\widetilde{X}$ is
      $\det$-$L^2$-acyclic and $\rho^{(2)}(\widetilde{X})$ vanishes;

    \end{enumerate}

  \item\label{list:main_properties_of_rho2(widetildeX):mapping_tori}
    \emph{Mapping tori}, see~\cite[Theorem~1.39 on page~42]{Lueck(2002)}.\\[1mm]
    Let $f \colon X \to X$ be a self homotopy equivalence of a finite $CW$-complex.  Denote
    by $T_f$ its mapping torus.

    \begin{enumerate}
    \item Then $\widetilde{T_f}$ is $L^2$-acyclic;

    \item If $\widetilde{X}$ is $\det$-$L^2$-acyclic, then $\rho^{(2)}(\widetilde{T_f})$ vanishes;
    \end{enumerate}

  \item\label{list:main_properties_of_rho2(widetildeX):hyperbolic}
    \emph{Hyperbolic manifolds}, see~\cite{Hess-Schick(1998)},~\cite[Theorem~1.39 on page~42]{Lueck(2002)}.\\[1mm]
    Let $M$ be a hyperbolic closed manifold of dimension $d$.
    \begin{enumerate}
    \item If $d$ is odd, $\widetilde{M}$ is $\det$-$L^2$-acyclic;

    \item Suppose that $d = 2m$ is even. Then $b_n^{(2)}(\widetilde{M})$ vanishes for
      $n \not= m$, and we have $(-1)^m \cdot \chi(M) = b_m^{(2)}(\widetilde{M})> 0$;

      \item For every number $m$ there exists an explicit constant $C_m> 0$ with the following
      property: If $M$ is a hyperbolic closed manifold of dimension $(2m+1)$ with volume
      $\vol(M)$, then
      \[\rho^{(2)}(\widetilde{M}) = (-1)^m \cdot C_m \cdot \vol(M).
      \]
      We have $C_1 = \frac{1}{6\pi}$. The number $\pi^m \cdot C_m$ is always rational;
    \end{enumerate}

  \item\label{list:main_properties_of_rho2(widetildeX):approximation} \emph{Approximation of
      $L^2$-Betti numbers by classical Betti numbers},
    see~\cite{Lueck(1994c)},\cite[Chapter~13]{Lueck(2002)}.
    \\[1mm]
    Let $X$ be a connected finite $CW$-complex with fundamental group \linebreak $G = \pi_1(X)$.
    Suppose that $G$ comes with a descending chain of subgroups
    \[
      G = G_0 \supseteq G_1 \supseteq G_2 \supseteq \cdots
    \]
    such that $G_i$ is normal in $G$, the index $[G:G_i]$ is finite, and we have
    $\bigcap_{i \ge 0} G_i = \{1\}$.

    Then $G_i\backslash \widetilde{X} \to X$ is a finite $[G:G_i]$-sheeted covering, and we
    get for $n \ge 0$
    \[
      b_n^{(2)}(\widetilde{X} ) = \lim_{i \to \infty} \frac{b_n(G_i\backslash
        \widetilde{X})}{[G:G_i]},
    \]
    where $b_n(G_i\backslash \widetilde{X})$ is the classical $n$th Betti number of the finite
    $CW$-complex $G_i\backslash \widetilde{X}$.
  \end{enumerate}
  \end{theorem}

  We refer for more information about $L^2$-invariants and their applications to algebra,
  geometry, group theory, index theory, operator algebras, topology, and $K$-theory for
  instance to~\cite{Kammeyer(2019)} and~\cite{Lueck(2002)}. We will discuss some of them later in this article.


  \typeout{------------------------ Section 3: The Atiyah Conjecture  --------------------------------}

  \section{The Atiyah Conjecture}\label{sec:The_Atiyah_Conjecture}


  In this section we discuss one of the most prominent conjectures in the field of $L^2$-invariants, the Atiyah Conjecture.
  It predicts the possible values of $L^2$-betti numbers and has many interesting implications.
  For example Kaplansky's Zero Divisor Conjecture for rational group rings follows from it.

  \subsection{Statement of the Atiyah Conjecture}%
  \label{subsec:Statement_of_the_Atiyah_Conjecture}

  \begin{conjecture}[Atiyah Conjecture]\label{con:Atiyah_Conjecture}
    Consider a field $F$ with $\IQ \subseteq F \subseteq \IC$.
    Let $G$ be a group possessing an upper bound on the orders of its finite subgroups.
    Let $\lcm(G)$ be the natural number given by the least common multiple of the orders of its finite subgroups.

    We say that  $G$ satisfies the \emph{Atiyah Conjecture with coefficients in $F$},
   if for any  finitely presented $FG$-module $M$ the von Neumann dimension satisfies
   \[\lcm(G) \cdot \dim_{\caln(G)}(\caln(G) \otimes_{FG} M) \in \IZ.
   \]

  \end{conjecture}

  Note that $G$ is torsionfree if and only if $\lcm(G) = 1$, so the Atiyah conjecture in this case predicts integrality of $L^2$-Betti numbers.
  For a field $F$, let us denote by $\cala_F$ the collection of all groups which admit a bound on the order of their finite subgroups and satisfy the Atiyah conjecture with coefficients in $F$.
   Obviously
  $\cala_{\IC} \subseteq \cala_F \subseteq \cala_{\IQ}$.

  \begin{remark}[Equivalent formulations of the Atiyah Conjecture~\ref{con:Atiyah_Conjecture}]%
    \label{rem:equivalent_formulation_of_the_Atiyah_Conjecture}
    Consider a field $F$ with $\IQ \subseteq F \subseteq \IC$.
    Let $G$ be a group that possesses an upper bound on the orders of is
    finite subgroups. Then the  following assertions are equivalent:
    \begin{itemize}
    \item[(1)]
    We have $G \in \cala_F$;

  \item[(2)]
    For any matrix $A \in M_{m,n}(FG)$ the dimension $\dim_{\caln(G)}(\ker(r_A))$ of
    the kernel of the $\caln(G)$-homomorphism $r_A\colon \caln(G)^m \to \caln(G)^n$ given by
    right multiplication with $A$ satisfies $\lcm(G) \cdot \dim_{\caln(G)}(\ker(r_A)) \in \IZ$;

  \item[(3)]
    For any $F G$-chain complex $C_*$ of finitely generated free $F G$-modules and all $n \in \IZ$ we have
    \[
   \lcm(G) \cdot  \dim_{\caln(G)}(H_n(\caln(G) \otimes_{F G} C_*)) \in \IZ.
  \]

  \end{itemize}

  If we additionally assume $F = \IQ$, then the
  three assertions above and the following assertion are equivalent:

  \begin{itemize}

  \item[(4)]
    For any $\IZ G$-chain complex $C_*$ of finitely generated free $\IZ G$-modules and all $n \in \IZ$ we have
    \[
   \lcm(G) \cdot   \dim_{\caln(G)}(H_n(\caln(G) \otimes_{\IZ G} C_*)) \in \IZ.
  \]
  \end{itemize}

  If we additionally assume $F = \IQ$ and  that $G$ is  finitely presented, then the
  four  assertions above and the following assertion are equivalent:

  \begin{itemize}
  \item[(5)] For any closed manifold $M$ with $\pi_1(M) \cong G$ and any $n \ge 0$ we have
  \[
  \lcm(G) \cdot  b_n^{(2)}(\widetilde{M}) \in \IZ.
   \]
  \end{itemize}
  All of these claims follow from from~\cite[Lemma~10.5 on page~371 and Lemma~10.7 on
  372]{Lueck(2002)}, the equality $\dim_{\caln(G)}(\caln(G)) = 1$, and the Additivity of the
  dimension function $\dim_{\caln(G)}$.
  \end{remark}

  The Atiyah Conjecture~\ref{con:Atiyah_Conjecture} is rather surprizing in view
  of~\eqref{L_upper_2-Betti_number_and_heat_kernel} and assertion~(5) appearing in
  Remark~\ref{rem:equivalent_formulation_of_the_Atiyah_Conjecture}, since there seems to be
  no reason why the expression appearing on the right hand side
  of~\eqref{L_upper_2-Betti_number_and_heat_kernel} should be an integer if $\pi_1(M)$ is
  torsionfree.


  \subsection{Status of the Atiyah Conjecture}%
  \label{subsec:Status_of_the_Atiyah_Conjecture}

  The notions of elementary amenable groups and amenable groups are explained for instance
  in~\cite[Subsection~6.4.1]{Lueck(2002)}.

  \begin{definition}[Class of groups $\calc$ ]\label{not:class_of_groups_calc}
    Let $\calc$ be the smallest class of groups satisfying the following conditions:

    \begin{enumerate}
    \item $\calc$ contains all free groups;

    \item If $\{G_i \mid i \in I\}$ is a directed system of subgroups directed by inclusion
      such that each $G_i$ belongs to $\calc$, then $G = \bigcup_{i \in I} G_i$ belongs to $\calc$;

  \item Let $1 \to K \to G \to Q \to 1$ be an extension of groups such that
    $H$ belongs to $\calc$ and $Q$ is elementary amenable, then $G$ belongs to $\calc$.

  \end{enumerate}
\end{definition}

  \begin{definition}[Class of groups $\cald$ ]\label{not:class_of_groups_cald}

    Let $\cald$ be the smallest class of groups satisfying the following conditions:

    \begin{enumerate}

    \item\label{not:class_of_groups_cald_trivial} The trivial group belongs to $\cald$;

    \item\label{not:class_of_groups_cald:colimits} If $\{G_i : i \in I\}$ is a filtered
      system of groups in $\cald$ (with arbitrary structure maps), then its colimit again
      belongs to $\cald$;

    \item\label{not:class_of_groups_cald:limits} If $\{G_i : i \in I\}$ is a cofiltered
      system of groups in $\cald$ (with arbitrary structure maps), then its limit again
      belongs to $\cald$;

    \item\label{not:class_of_groups_cald;subgroups} If $G$ belongs to $\cald$ and
      $H \subseteq G$ is a subgroup, then $H \in \cald$;

    \item\label{not:class_of_groups_cald_elem_amen} If $p\colon G \to A$ is an epimorphism
      of a torsionfree group $G$ onto an elementary amenable group $A$ and if
      $p^{-1}(B) \in \cala_F$ for every finite group $B \subset A$, then $G \in \cala_F$.

      \end{enumerate}
  \end{definition}

  Note that each element in $\cald$ is a torsionfree group, and the class $\cald$ contains
  all residually (torsionfree elementary amenable) groups.

  A group is called \emph{locally indicable}, if every non-trivial finitely generated
  subgroup admits an epimorphism onto $\IZ$. Locally indicable groups are torsionfree.
  Examples for locally indicable groups are one-relator groups.

  \begin{theorem}[Status of the Atiyah Conjecture~\ref{con:Atiyah_Conjecture} with
    coefficients in $F$]%
    \label{the:status_of_Atyah_Conjecture} Consider a field $F$ with
    $\IQ \subseteq F \subseteq \IC$.

    \begin{enumerate}

    \item\label{the:Status_of_the_Atiyah_Conjecture:Linnell} If $G$ belongs to $\calc$ and
      possesses an upper bound on the orders of is finite subgroups, then $G \in \cala_F$;

    \item\label{the:status_of_Atyah_Conjecture:cald}
     If $G$ belongs to $\cald$, then $G$ is torsionfree and $G \in \cala_F$;

   \item\label{the:status_of_Atyah_Conjecture:locally_indicable} If
     $1 \to H \to G \to Q \to 1$ is an extension of groups, $H$ is torsionfree and belongs
     to $\cala_F$, and $Q$ is locally indicable, then $G$ is torsionfree and belongs to
     $\cala_F$;

    \item\label{the:status_of_Atyah_Conjecture:subgroups} If $G \in\cala_F$ and
      $H \subseteq G$ is a subgroup, then
      $H \in \cala_F$;

   \item\label{the:status_of_Atyah_Conjecture:directed_unions}
     If G is the directed union $\bigcup_{i \in I} G_i$ of subgroups $G_i$ directed by inclusion
     and each $G_i$ belongs to $\cala_F$, then $G$ belongs to $\cala_F$;

   \item\label{the:status_of_Atyah_Conjecture:fin_gen} The group $G$ belongs to $\cala_F$
     if and only if all its finitely generated subgroups belong to $\cala_F$;

  \item\label{the:status_of_Atyah_Conjecture:finite_quotients}
   If $1 \to K \to G \to Q \to 1$ is an
      extension of groups such that $K$ is finite and $G$ belongs to $\cala_F$, then $Q$
      belongs to $\cala_F$;

   \item\label{the:status_of_Atyah_Conjecture:pi_low} Let $M$ be a connected (not
     necessarily compact) $d$-dimensional manifold (possibly with non-empty boundary) such
     that $d \le 3$ and its fundamental group $\pi_1(M)$ is torsionfree, then
     $\pi_1(M) \in \calc$ and hence $\pi_1(M) \in \cala_F$;

   \item\label{the:status_of_Atyah_Conjecture:special_groups} If the group $G$ possesses an
     upper bound on the orders of its finite subgroups and belongs to one of the following
     classes below, then $G$ belongs to $\cala_F$:
     \begin{enumerate}
     \item Residually \{torsionfree elementary amenable\} groups;
     \item Free by elementary amenable groups;
     \item Braid groups;
     \item Right-angled Artin and Coxeter groups;
     \item Torsionfree $p$-adic analytic pro-$p$-groups;
     \item Locally indicable groups;
     \item One-relator groups.
     \end{enumerate}
   \end{enumerate}
  \end{theorem}
  \begin{proof}~\eqref{the:Status_of_the_Atiyah_Conjecture:Linnell} This is due to Linnell,
    see for instance~\cite{Linnell(1993)} or~\cite[Theorem~10.19 on page~378]{Lueck(2002)}.
    \\[1mm]~\eqref{the:status_of_Atyah_Conjecture:cald}
    This follows from~\cite[Corollary~1.2]{Jaikin-Zapirain(2019)}, which is based on
    on~\cite[Theorem~1.4]{Dodziuk-Linnell-Mathai-Schick_Yates(2003)}.
    \\[1mm]~\eqref{the:status_of_Atyah_Conjecture:locally_indicable} This follows
    from~\cite[Proposition~6.5]{Jaikin-Zapirain+Lopez-Alvarez(2020)}.
    \\[1mm]~\eqref{the:status_of_Atyah_Conjecture:subgroups} This follows
    from~\cite[Theorem~6.29~(2) on page~253]{Lueck(2002)}.
    \\[1mm]~\eqref{the:status_of_Atyah_Conjecture:directed_unions}
    See~\cite[Lemma~10.4 on  page~371]{Lueck(2002)}.
    \\[1mm]~\eqref{the:status_of_Atyah_Conjecture:fin_gen} This
    follows from assertions~\eqref{the:status_of_Atyah_Conjecture:subgroups}
    and~\eqref{the:status_of_Atyah_Conjecture:directed_unions}.
    \\[1mm]~\eqref{the:status_of_Atyah_Conjecture:finite_quotients} This follows
    from\cite[Lemma~13.45 on page~473]{Lueck(2002)}.
    \\[1mm]~\eqref{the:status_of_Atyah_Conjecture:pi_low} This follows
    from~\cite[Theorem~1.1]{Kielak-Linton(2023Atiyah)} for $d = 3$. The case $d = 2$ can be
    reduced to the case $d =3$ by crossing with $S^1$ and
    assertion~\eqref{the:status_of_Atyah_Conjecture:subgroups}.
    \\[1mm]~\eqref{the:status_of_Atyah_Conjecture:special_groups} This follows from other assertions
    or from~\cite[Theorem~1.1 and Corollary~1.2]{Jaikin-Zapirain(2019)} using~\cite[Theorem~2]{Linnell-Okun-Schick(2012)}
    and~\cite[Theorem~1.1]{Farkas-Linnell(2006)}.
  \end{proof}

  \begin{remark}\label{rem:Atiyah_for_hyperbolic_groups}
    The class $\cala_F$ is very large by aforementioned results. Nevertheless, we do not know whether the
    Atiyah Conjecture~\ref{con:Atiyah_Conjecture} holds for all hyperbolic groups
    or for all amenable groups.
  \end{remark}

  There are partial results on the difficult question, whether the
  Atiyah Conjectures~\ref{con:Atiyah_Conjecture} holds for a group $G$ if it holds for a subgroup of finite
  index, see for instance~\cite{Linnell-Schick(2007)}.

  \begin{remark}\label{rem:Atiyah_without_bound_on_htr_orders_of_finite_subgroups}
  	Dropping the condition on an upper bound on the orders of finite subgroups, one might still ask if the $L^2$-Betti numbers are always rational. This goes back to Atiyah's original question~\cite[page~72]{Atiyah(1976)}, who asked for rationality of $L^2$-Betti numbers of closed manifolds.
    Austin~\cite[Corollary~1.2]{Austin(2013)} gave the first example of a finitely generated
    group $G$, where for some matrix $A \in M_{m,n}(\IQ G)$ the dimension
    $\dim_{\caln(G)}(\ker(r_A))$ of the kernel of the $\caln(G)$-homomorphism
    $r_A\colon \caln(G)^m \to \caln(G)^n$ is
    irrational. Grabowski~\cite[Theorem~1.3]{Grabowski(2014Turing)} proved, using Turing
    machines, that any non-negative real number can occur in this way for some finitely
    generated group $G$ and some matrix $A$.  L\"oh and Uschold~\cite{Loeh-Uschold(2022)}
    investigate the computability degree of real numbers arising as $L^2$-Betti numbers or
    $L^2$-torsion of groups, parametrised over the Turing degree of the word problem.
    Roughly speaking, the complexity of the computation of $L^2$-invariants of a group is the same as the complexity of the word problem.  This is due to the combinatorial computation in terms of characteristic sequences mentioned in Remark~\ref{rem:combinatorial_computation}.
  \end{remark}


  \subsection{Embedding the group ring of a torsionfree group into a skewfield}%
  \label{subsec:Embedding_the_group_ring_of_a_torsionfree_group_into_a_skewfielde}

  Associated to the von Neumann algebra $\caln(G)$ is the algebra of affiliated operators
  $\calu(G)$ which contains $\caln(G)$.  It can be defined analytically or just as the Ore
  localization of $\caln(G)$ with respect to the multiplicative subset of non-zero
  divisors.  Now one can consider the so called division closure $\cald(G)$ of $\IC G$ in
  $\calu(G)$, i.e., the smallest ring $\IC G \subset \cald (G) \subset \calu(G)$ such that
  if $x \in \cald(G)$ is invertible in $\calu(G)$, its inverse is already contained in
  $\cald(G)$.

  The proof of the following is based on ideas of Peter Linnell from~\cite{Linnell(1993)}
  which have been elaborated and expanded on in~\cite[Theorem~8.29 on page~330 and Lemma~10.39 on
  page~388]{Lueck(2002)} and~\cite{Reich(2006)}, see
  also~\cite[Theorem~3.8~(1) and~{2}]{Friedl-Lueck(2019Thurston)}.  The proofs given in the references above
  for $F = \IC$ and $F = \IQ$ carry directly over to arbitrary $F$.

  \begin{theorem}[Main properties of $\cald(G)$]\label{the:Main_properties_of_cald(G)}
    Let $G$ be a torsionfree group. Consider a field $F$ with $\IQ \subseteq F \subseteq \IC$.

    \begin{enumerate}

    \item\label{the:Main_properties_of_cald(G):skew_field} The group $G$ satisfies the
      Atiyah Conjecture with coefficients in $F$ if and only if $\cald(G)$ is a skew
      field;

    \item\label{the:Main_properties_of_cald(G):dim} Suppose that $G$ satisfies the Atiyah
      Conjecture.  Let $C_*$ be a projective $F G$-chain complex. Then we get for all $n \ge 0$
      \[
        b_n^{(2)}\bigl(\caln(G) \otimes_{FG} C_*\bigr) =
        \dim_{\cald(G)}\bigl(H_n(\cald(G) \otimes_{\IQ G} C_*)\bigr).
      \]
      In particular $b_n^{(2)}\bigl(\caln(G) \otimes_{F G} C_*\bigr)$ is an integer or $\infty$.

    \end{enumerate}
  \end{theorem}

  Theorem~\ref{the:Main_properties_of_cald(G)} shows that the Atiyah
  Conjecture~\ref{con:Atiyah_Conjecture} is related to the question whether for a
  torsionfree group $G$ the group ring $FG$ can be embedded into a skew field, see for
  instance~\cite{Henneke-Kielak(2021)}.  Note  that the existence of
  an embedding of $FG$ into a skewfield implies that
  $FG$ has no non-trivial zero-divisors which is predicted by the  Zero-Divisor-Conjecture
  of Kaplansky.


  \subsection{The Algebraic Atiyah Conjecture}%
  \label{subsec:The_Algebraic__Atiyah_Conjecture}

  Recall that Linnell's program to approach the Atiyah Conjecture consists of a $K$-theoretic part and a ring theoretic part, namely that
 the composite
 \begin{equation*}
   \colim_{H \subseteq G, |H| < \infty} K_0(FH) \to K_0(FG) \to K_0(\cald(G))
 \end{equation*}
 is surjective and
 that  $\cald(G)$ is semisimple, see~\cite[Lemma~10.28 on page~382 and Theorem~10.38 on page~387]{Lueck(2002)}.
 There is the following
  so called algebraic Atiyah Conjecture, which is a purely $K$-theoretic statement
  and taken from~\cite[Conjecture~7.2]{Jaikin-Zapirain(2017positive)},
  c,f.~\cite[Lemma~10.26 on page~382]{Lueck(2002)}. It implies the
  Atiyah Conjecture~\ref{con:Atiyah_Conjecture}, see~\cite[Lemma~10.26 on page~382]{Lueck(2002)}.

  \begin{conjecture}[Algebraic Atiyah Conjecture]\label{con_algebraic_Atiyah_Conjecture}
    Let $G$ be a group and let $F$  be a field $F$ with $\IQ \subseteq F \subseteq \IC$,
    which is closed under complex conjugation.
    Let $\calr_{FG}$ be the $\ast$-regular  closure of $FG$ in $\calu(G)$.

    We say that  $G$ satisfies the \emph{Algebraic Atiyah Conjecture with coefficients in $F$}
    if the composite of the canonical maps
    \[
    \colim_{H \subseteq G, |H| < \infty} K_0(FH) \to K_0(FG) \to K_0(\calr_{FG})
    \]
    is surjective.
  \end{conjecture}

  Note that the Farrell-Jones Conjecture, which is known for a large class of groups, see
  for instance~\cite[Theorem~12.56 and Theorem~15.1]{Lueck(2022book)}, predicts the
  bijectivity of the map $\colim_{H \subseteq G, |H| < \infty} K_0(FH) \to K_0(FG)$. If the
  Farrell-Jones Conjecture holds, then the Algebraic Atiyah
  Conjecture~\ref{con_algebraic_Atiyah_Conjecture} is equivalent to the claim that the map
  $K_0(FG) \to K_0(\calr_{FG})$ is surjective.  For a group $G$, for which there is a bound
  on its finite subgroups, the algebraic Atiyah
  Conjecture~\ref{con_algebraic_Atiyah_Conjecture} is equivalent to the so called
  center-valued Atiyah Conjecture with coefficients in $F$, which implies the Atiyah
  Conjecture~\ref{con:Atiyah_Conjecture} with coefficients in $F$,
  see~\cite[Definition~2.5.1, Proposition~2.5.2, Theorem~3.1.4]{Henneke(2021PhD)} and
  also~\cite[Definition~1.2, Proposition~1.3]{Knebusch-Linnell-Schick(2017)}.

  For more information about the Atiyah Conjecture we refer for
  instance to~\cite[Chapter~10]{Lueck(2002)}.


  \typeout{----------------------------- Section 4: The Singer Conjecture  ---------------------------}

  \section{The Singer Conjecture}\label{sec:The_SingerConjecture}


  We now turn attention to a series of conjectures about $L^2$-invariants of aspherical manifolds.
  They deal with the phenomenon that $L^2$-invariants often vanish outside of the middle dimension.

  \subsection{Statement of the Singer Conjecture}%
  \label{subsec:Statement_of_the_Singer_Conjecture}

  \begin{conjecture}[Singer Conjecture]\label{con:Singer_Conjecture}
    If $M$ is an aspherical closed topological  manifold, then we get for $n \ge 0$
    \[
      b_n^{(2)}(\widetilde{M}) = 0 \quad \text{if} \; 2n \not= \dim(M).
    \]
    If $M$ is an aspherical closed topological  manifold of even dimension $2m$ , then
    \[
      (-1)^m \cdot \chi(M) = b_m^{(2)}(\widetilde{M}) \ge  0.
    \]
    If $M$ is a closed connected smooth manifold of even dimension $2m$ admitting a Riemannian metric of negative sectional curvature, then
    \[
      (-1)^m \cdot \chi(M) = b_m^{(2)}(\widetilde{M}) > 0.
    \]
  \end{conjecture}

  The equality $(-1)^m \cdot \chi(M) = b_m^{(2)}(\widetilde{M})$ appearing in the
  Singer Conjecture~\ref{con:Singer_Conjecture} above follows from the the Euler-Poincar\'e
  formula $\chi(M) = \sum_{p \ge 0} (-1)^p \cdot b_p^{(2)}(\widetilde{M})$.

  The Singer Conjecture~\ref{con:Singer_Conjecture} is consistent with the Atiyah
  Conjecture in the sense that it predicts that the $L^2$-Betti numbers
  $b_n^{(2)}(\widetilde{M})$ for an aspherical closed manifold $M$ are all integers.

  In original versions of the Singer Conjecture~\ref{con:Singer_Conjecture}
  the condition aspherical closed manifolds was
  replaced by the condition closed Riemannian manifold with non-positive sectional
  curvature. Note that a closed Riemannian manifold with non-positive sectional curvature is
  aspherical by Hadamard's Theorem.


  \subsection{Hopf Conjectures}%
  \label{subsec:Hopf_Conjectures}

  Obviously the
  Singer Conjecture~\ref{con:Singer_Conjecture} implies the following conjecture in the
  cases, where $M$ is aspherical or has negative sectional curvature.

  \begin{conjecture}[Hopf Conjecture]\label{con:Hopf_Conjecture}
    If $M$ is an aspherical closed topological manifold of even dimension $\dim (M) = 2m$, then
    \[
      (-1)^{m} \cdot \chi(M) \ge 0.
    \]
    If $M$ is a closed smooth manifold of even dimension $\dim (M) = 2m$ with Riemannian metric
    and with sectional curvature $\sec(M)$, then
    \[
      \begin{array}{rlllllll}
        (-1)^m \cdot \chi(M) & > & 0 & &
                                         \text{if}   & \sec(M) & < & 0;
        \\
        (-1)^m \cdot \chi(M) & \ge   & 0 & &
                                             \text{if}   & \sec(M) & \le & 0;
        \\
        \chi(M) & = & 0 & &
                            \text{if}  & \sec(M) & = & 0;
        \\
        \chi(M) & \ge & 0 & &
                              \text{if}   & \sec(M) & \ge & 0;
        \\
        \chi(M) & > & 0 & &
                            \text{if}  & \sec(M) & > & 0.
      \end{array}
    \]
  \end{conjecture}

  The following version of the Hopf Conjecture for $L^2$-torsion appears in~\cite[Conjecture~11.3 on page~418]{Lueck(2002)}

  \begin{conjecture}[Hopf Conjecture for $L^2$-torsion]\label{con:Hopf_Conjecture_for_L2-torsion}
    If $M$ is an aspherical closed topological manifold of odd dimension $\dim (M) = 2m+1$,
    then $M$ is $\det$-$L^2$-acyclic, and its $L^2$-torsion satisfies
    \[
      (-1)^{m} \cdot \rho^{(2)}(\widetilde{M})\ge 0.
    \]
    If $M$ is a closed smooth manifold of odd dimension $\dim (M) = 2m+1$ with Riemannian metric
    and with sectional curvature $\sec(M) < 0$, then $M$ is $\det$-$L^2$-acyclic, and its $L^2$-torsion satisfies
    \[
      (-1)^{m} \cdot \rho^{(2)}(\widetilde{M}) > 0.
    \]
    If $M$ is an aspherical closed topological manifold of odd dimension $\dim (M) = 2m+1$
    whose fundamental group contains an amenable infinite normal subgroup,
    then $M$ is $\det$-$L^2$-acyclic and its $L^2$-torsion satisfies
    \[
      \rho^{(2)}(\widetilde{M}) =  0.
    \]

  \end{conjecture}

  The Hopf Conjecture~\ref{con:Hopf_Conjecture_for_L2-torsion} for $L^2$-torsion is known to
  be true if one of the following conditions is satisfied:
  \begin{enumerate}
  \item $\dim(M) \le 3$;
  \item $M$ is a locally symmetric space;
  \item $\pi_1(M)$ contains an elementary amenable infinite normal subgroup;
  \item $M$ carries a non-trivial $S^1$-action.
  \end{enumerate}

  Statement (1) follows from combining~\cite[Theorem~0.7]{Lueck-Schick(1999)} with  Thurston's Geometrization conjecture, which
  is known to be true
  by~\cite{Kleiner-Lott(2008),Morgan-Tian(2014)} following the spectacular outline of Perelman.
    Statement (2) can be found in~\cite[Corollary~5.16 on page~231]{Lueck(2002)} and statements (3) and (4) are from~\cite[Corollary~1.13]{Lueck(2013l2approxfib)}. See also
  Remark~\ref{rem:motivation_Modifiied_Conjecture}.


  \subsection{Status of the Singer Conjecture}%
  \label{subsec:Status_of_the_Singer_Conjecture}

    The Singer Conjecture~\ref{con:Singer_Conjecture} is known for an aspherical closed
    smooth manifold $M$ in the following cases:
    \begin{enumerate}
    \item  $\dim(M) \le 3$;
    \item  $M$ comes with a Riemannian metric whose sectional curvature
      is negative and satisfies certain pinching conditions;
    \item $M$ is a locally symmetric space;
    \item $M$ possesses  a Riemannian metric, whose sectional curvature
      is negative and $M$ carries some K\"ahler structure;
      \item  $M$ is an aspherical closed K\"ahler manifold, whose fundamental group is
        word-hyperbolic in the sense of Gromov~\cite{Gromov(1987)}.
      \item $\pi_1(M)$ contains an amenable infinite normal subgroup.
      \item $M$ carries a non-trivial $S^1$-action;
      \item $M$ fibers over $S^1$.
  \end{enumerate}
  The precise statements and proofs and the relevant references,
  e.g.,~\cite{Ballmann-Bruening(2001), Cheeger-Gromov(1986), Gromov(1991), Jost-Xin(2000),
    Lueck(1994b), Olbrich(2002)}, can be found in~\cite[Theorem~1.39 on page~42,
  Corollary~1.43 on page~48, Theorem~1.44 on page~48, Section~11.1]{Lueck(2002)}, again
  using Thurston's Geometrization Conjecture for the three dimensional case.  We also
  mention that proofs of (1) - (5) use geometric properties of the manifold whereas (6) -
  (8) use algebraic topological techniques and do not use the manifold structure.

  Partial results about the Singer Conjecture~\ref{con:Singer_Conjecture} for right-angled
  Coxeter groups can be found in Davis-Okun~\cite{Davis-Okun(2001)}.

  The paper by Albanese, Di Cerbo and Lombardi~\cite{Albanese-Di-Cerbo-Lombardi(2023)}
  deals with the Singer Conjecture for aspherical complex surfaces and
  proves it for aspherical complex surfaces with residually finite
  fundamental groups in~\cite[Theorem~1.5]{Albanese-Di-Cerbo-Lombardi(2023)}.

  In contrast to the Atiyah Conjecture, evidence for the Singer
  Conjecture~\ref{con:Singer_Conjecture} comes from computations only and no promising
  proof strategy is known.  In some sense Poincar\'e duality together with $L^2$-bounds on
  differential forms on $\widetilde M$ seem to force the $L^2$-Betti numbers
  $b_p^{(2)}(\widetilde{M})$ of an aspherical closed manifold to concentrate in the middle
  dimension.  One may wonder what happens if we replace $M$ by an aspherical finite
  Poincar\'e complex in the Singer Conjecture~\ref{con:Singer_Conjecture}.  There are
  counterexamples to the Singer Conjecture~\ref{con:Singer_Conjecture} if one weakens
  aspherical to rationally aspherical, see~\cite[Theorem~4]{Avramidi(2018)}. The reader
  should also take a look at Remark~\ref{rem:Singer_and_growth} and
  Remark~\ref{rem:motivation_Modifiied_Conjecture}.


  \subsection{The proper Singer Conjecture and the action dimension}%
  \label{subsec:The_proper_Singer_Conjecture_and_the_action_dimension}

    One may generalize the Singer Conjecture~\ref{con:Singer_Conjecture}
    to the following

    \begin{conjecture}[Proper Singer Conjecture]\label{con_proper_Singer_Conjecture}
      A group $G$ satisfies the \emph{proper Singer Conjecture}, if for any contractible
      topological manifold $M$ without boundary on which $G$ acts properly and cocompactly,
      we have for all $n \ge 0$
      \[
        b_n^{(2)}(M;\caln(G)) = 0 \quad \text{if} \; 2n \not= \dim(M).
      \]
    \end{conjecture}
    Note that the proper Singer Conjecture is a statement about a group $G$, whereas the
    Singer Conjecture~\ref{con:Singer_Conjecture} is a statement about an aspherical closed
    manifold. Provided that $G$ is torsionfree, the proper Singer Conjecture holds for $G$
    if and only if the Singer Conjecture~\ref{con:Singer_Conjecture} holds for one (and hence all)
    aspherical closed manifold $M$ with $\pi_1(M) \cong G$.

    The proper Singer Conjecture for a group $G$ is equivalent to the cadim-Con\-jec\-ture (cadim standing for compact action dimension)
    for $G$ for manifolds with PL-boundary, see~\cite[Theorem~4.10]{Okun-Schreve(2016)}.
    The latter predicts that for a contractible topological manifold $M$ (possibly with boundary)
    that admits a cocompact proper topological action of $G$ such that the boundary
    $\partial M$ of $M$ carries a $\PL$-structure for which the $G$-action on $\partial M$
    is through $\PL$-automorphisms we have
  \[
    2 \cdot \sup\{n \in \IZ^{\ge 0} \mid b_n^{(2)}(G) \not= 0\} \le  \dim(M).
  \]
  For more informations about the notion of the  (compact) action dimension and its
  relationship  to $L^2$-Betti numbers we refer for instance
  to~\cite{Avramidi-Davis-Okun-Schreve(2016), Bestvina-Kapovich-Kleiner(2002), Okun-Schreve(2016)}.


  \subsection{The $\IF_p$-Singer Conjecture}\label{subsec:The_F_p-Singer_Conjecture}
  One may also consider the \emph{$\IF_p$-Singer Conjecture}, which predicts that
  for an aspherical closed topological  manifold $M$ and any prime $p$ we get for  $n \ge 0$
  \[
    b_n^{(2)}(\pi_1(M);\IF_p) = 0 \quad \text{if} \; 2n \not= \dim(M).
  \]
  Here $b_n^{(2)}(\pi_1(M);\IF_p)$ is defined only if $\pi_1(M)$ is residually finite.
  Namely, for a chain of normal subgroups of finite index
  $\pi_1(M) = \Gamma_0 \supseteq \Gamma_1 \supseteq \Gamma_2 \supseteq \cdots$ with
  $\bigcap_{n = 0}^{\infty} \Gamma_n = \{1\}$, one puts
  $b_n^{(2)}(\pi_1(M);\IF_p) = \limsup_{n \to \infty} \frac{b_n(\Gamma_k;\IF_p)}{[\Gamma:  \Gamma_k]}$.
  It is unknown in general  whether this definition depends on the chain of subgroups  $\Gamma_k$ and the statement above is to be understood in the sense that it holds for any such chain.
  The definition of $L^2$-Betti numbers over $\IF_p$ is motivated by the
  Approximation Theorem in characteristic zero~\ref{thm:main_properties_of_rho2(widetildeX)}~\eqref{list:main_properties_of_rho2(widetildeX):approximation}.
  The $\IF_p$-Singer Conjecture is open for $\dim(M) = 3$. However, for any odd prime $p$,
  the $\IF_p$-Singer Conjecture fails in all odd dimensions larger or equal than $7$ and all even dimensions
  larger or equal than $14$, see~\cite[Theorem~4]{Avramidi-Okun-Schreve(2021)}.


  \typeout{---------------------------- Section 5: Determinant Conjecture
    ---------------------------}

  \section{The Determinant Conjecture}\label{sec:The_Determinant-Conjecture}

  Next we explain the Determinant Conjecture~\cite[Conjecture~13.2 on
  page~454]{Lueck(2002)}.
  It is needed for the definition of $L^2$-determinant  class and implies homotopy invariance of $L^2$-torsion,
  see Theorem~\ref{thm:main_properties_of_rho2(widetildeX)}~\eqref{list:main_properties_of_rho2(widetildeX):homotopy_invariance}.
  Furthermore, it implies the Approximation Conjecture, see
  Remark~\ref{rem:The_Determinant_Conjecture_implies_the_Approximation_Conjecture_for_L2-Betti_numbers}.

  \begin{conjecture}[Determinant Conjecture for a group $G$]\label{con:Determinant_Conjecture}\ \\
    For any matrix $A \in M_{r,s}(\IZ G)$, the Fuglede-Kadison determinant of the
    $\caln(G)$-homomorphism $r_A \colon \caln(G)^r \to \caln(G)^s$ given by right
    multiplication with $A$ satisfies
    \[
      {\det}_{\caln(G)}\left(r_A \right) \ge 1.
    \]
  \end{conjecture}

  The class $\calf$ of groups for which it is true has the following properties:
  \begin{enumerate}

  \item\label{rem:status_of_Determinant_Conjecture:trivial_group} \emph{Trivial group}.\\
    The trivial group belongs to $\calf$;

  \item\label{rem:status_of_Determinant_Conjecture:amenable_quotient}
    \emph{Amenable quotient}.\\
    Let $H \subset G$ be a normal subgroup. Suppose that $H \in \calf$ and the quotient
    $G/H$ is amenable. Then $G \in \calf$;

  \item\label{rem:status_of_Determinant_Conjecture:direct_limit}
  	\emph{Filtered colimits}.\\
  	The class $\calf$ is closed under countable filtered colimits;

  \item\label{rem:status_of_Determinant_Conjecture:inverse_limit}
   \emph{Cofiltered limits}.\\
   The class $\calf$ is closed under countable cofiltered limits;

  \item\label{rem:status_of_Determinant_Conjecture:subgroups}
    \emph{Subgroups}.\\
    If $H$ is isomorphic to a subgroup of a group $G$ with $G \in \calf$, then
    $H \in \calf$.

  \item\label{rem:status_of_Determinant_Conjecture:quotient_with_finite_kernel}
    \emph{Quotients with finite kernel}.\\
    Let $1 \to K \to G \to Q \to 1$ be an exact sequence of groups. If $K$ is finite and $G$
    belongs to $\calf$, then $Q$ belongs to $\calf$;

  \item\label{rem:status_of_Determinant_Conjecture:sofic_groups} \emph{Sofic groups}.\\
  Sofic groups belong to
    $\calf$.
  \end{enumerate}

  For the verification of the Determinant Conjecture for the groups above,
  see~\cite[Theorem~5]{Elek-Szabo(2005)},%
~\cite[Section~13.2 on pages~459~ff]{Lueck(2002)},~\cite[Theorem~1.21]{Schick(2001b)}.

  To sketch how large $\calf$ really is, let us mention that already the class of sofic groups is very large.  It is closed under direct and free products,
  taking subgroups, taking inverse and direct limits over directed index sets, and is closed
  under extensions with amenable groups as quotients and a sofic group as kernel.  In
  particular it contains all residually amenable groups.  One expects that there exists
  non-sofic groups but no example is known.  More information about sofic groups can be
  found for instance in~\cite{Elek-Szabo(2006)} and~\cite{Pestov(2008)}. More information
  about the Determinant Conjecture~\ref{con:Determinant_Conjecture} can be found
  in~\cite[Chapter~13]{Lueck(2002)} and~\cite{Schick(2001b)}.


  \typeout{------------------- Section 6: Approximation  for finite index  normal chains ------------------}

\section{Approximation for finite index normal chains}%
\label{sec:Approximation_for_finite_index_normal_chains}

Recall that the Approximation
Theorem~\ref{list:main_properties_of_rho2(widetildeX):approximation} predicts that
$L^2$-Betti numbers of spaces can be computed as limits of normalized ordinary Betti
numbers of certain towers of coverings. In this section we discuss this phenomenon for
more general invariants.


  \subsection{Basic setup for approximation for finite index normal chains}%
  \label{subsec:Basic_setup_for_approximation_for_finite_index_normal_chains}

  Let $G$ be a (discrete) group. A \emph{finite index normal chain} $\{G_i\}$  for $G$ is a  descending chain of subgroups
  \begin{eqnarray}
    & G = G_0 \supseteq G_1 \supseteq G_2 \supseteq \cdots &
    \label{finite_index_normal_chain}
  \end{eqnarray}
  such that $G_i$ is normal in $G$, the index $[G:G_i]$ is finite and $\bigcap_{i \ge 0} G_i = \{1\}$.

  Let $p \colon \overline{X} \to X$ be a $G$-covering. Put $X[i] := G_i\backslash \overline{X}$.
  We obtain a $[G:G_i]$-sheeted covering $p[i] \colon X[i] \to X$. Its total space $X[i]$
  inherits the structure of a finite $CW$-complex, a closed manifold, or a closed Riemannian
  manifold respectively if $X$ has the structure of a finite $CW$-complex, a closed manifold,
  or a closed Riemannian manifold respectively.

  Let $\alpha$ be a classical
  topological invariant such as the Euler characteristic, the signature, the $n$th Betti
  number with coefficients in the field $\IQ$ or $\IF_p$, torsion in the sense of
  Reidemeister or Ray-Singer,
  or the logarithm of the cardinality of the torsion subgroup of the $n$th homology group
  with integral coefficients. We want to study the sequence
  \begin{eqnarray*}
    \left(\frac{\alpha(X[i])}{[G:G_i]}\right)_{i \ge 0}.
    \label{alpha(X[i])/[G:G_i]}
  \end{eqnarray*}

  \begin{problem}[Approximation Problem]\label{pro:approximation_problem}\
    \begin{enumerate}

    \item Does the sequence converge?

    \item If yes, is the limit independent of the chain $\{G_i\}$?

    \item If yes, what is the limit?

    \end{enumerate}
  \end{problem}

  The hope is that the answer to the first two  questions is yes and the limit turns out to
  be an $L^2$-analogue $\alpha^{(2)}$ of $\alpha$ applied to the $G$-space $\overline{X}$, i.e., one can prove
  an equality of the type
  \begin{eqnarray}
   \lim_{i \to \infty} \frac{\alpha(X[i])}{[G:G_i]} = \alpha^{(2)}(\overline{X};\caln(G)).
  \label{expected_formula}
  \end{eqnarray}
  Here $\caln(G)$ stands for the group von Neumann algebra and is a reminiscence of the  fact that the $G$-action on
  $\overline{X}$ plays a role.
  Equation~\eqref{expected_formula} is often used to compute the $L^2$-invariant  $\alpha^{(2)}(\overline{X};\caln(G))$ by its
  finite-dimensional analogues $\alpha(X[i])$.
  On the other hand, it implies the existence of finite  coverings with large $\alpha(X[i])$, if  $\alpha^{(2)}(\overline{X};\caln(G))$
  is known to be positive.


  \subsection{The Euler characteristic}\label{subsec:Euler_characteristic}

  The \emph{Euler characteristic}
  $\chi(X)$ of a finite $CW$-complex is
  multiplicative under finite coverings. Because  this implies \linebreak $\chi(X) =
  \frac{\chi(X[i])}{[G:G_i]}$, the answer in this case is yes to the  questions
  appearing in Problem~\ref{pro:approximation_problem}, and the limit is
  \begin{eqnarray}
    \lim_{i \to \infty} \frac{\chi(X[i])}{[G:G_i]} & = & \chi(X).
    \label{solution_for_chi(X)}
  \end{eqnarray}


  \subsection{The Signature}\label{subsec:Signature}

  Next we consider the \emph{signature} $\sign(M)$ of a closed oriented topological $4k$-dimensional manifold $M$.
  It is known that it is multiplicative under finite coverings, however, the proof is more involved
  than the one for the Euler characteristic. It follows for instance from Hirzebruch's
  Signature Theorem, see~\cite{Hirzebruch(1970)}, or Atiyah's $L^2$-index
  theorem~\cite[(1.1)]{Atiyah(1976)} in the smooth case; for closed topological manifolds
  see Schafer~\cite[Theorem 8]{Schafer(1970)}.
  Since this implies $\sign(X) = \frac{\sign(X[i])}{[G:G_i]}$, each of the
  questions appearing in Problem~\ref{pro:approximation_problem} has a positive answer, and the limit is
  \begin{eqnarray}
    \lim_{i \to \infty} \frac{\sign(X[i])}{[G:G_i]} & = & \sign(X).
    \label{solution_for_sign(X)_manifolds}
  \end{eqnarray}

  The next level of generality is to pass from a oriented closed topological manifold to a oriented finite Poincar\'e
  complex, whose definition is due to Wall~\cite{Wall(1967)}.  For them the signature is
  still defined if the dimension is divisible by $4$.  There are Poincar{\'e} complexes $X$ for which
  the signature is not multiplicative under finite coverings,
  see~\cite[Example~22.28]{Ranicki(1992)},~\cite[Corollary~5.4.1]{Wall(1967)}.  Hence the
  situation is more complicated here.  Nevertheless, it turns out in this case each of the
  questions appearing in Problem~\ref{pro:approximation_problem} has a positive answer,  and the limit is
  \begin{eqnarray}
    \lim_{i \to \infty} \frac{\sign(X[i])}{[G:G_i]} & = & \sign^{(2)}(\overline{X};\caln(G)),
    \label{solution_for_signi(M)_Poincare}
  \end{eqnarray}
  where $\sign^{(2)}(\overline{X};\caln(G))$ denotes the $L^2$-signature, which is in general
  different from $\sign(X)$ for a finite Poincar\'e complex $X$.

  For more information and details we refer to~\cite{Lueck-Schick(2003),Lueck-Schick(2005)}.


  \subsection{Approximation of $L^2$-Betti numbers in characteristic zero}%
  \label{subsec:Approximation_of_L2_-Betti_numbers_in_characteristic_zero}

  Fix a field $F$ of characteristic zero. We consider the \emph{$n$th Betti number
    with $F$-coefficients} $b_n(X;F) := \dim_F(H_n(X;F))$.  Note that $b_n(X;F) =
  b_n(X;\IQ) = \rk_{\IZ}(H_n(X;\IZ))$ holds, where $\rk_{\IZ}$ denotes the rank of a finitely
  generated abelian group. In this case each of the
  questions appearing in Problem~\ref{pro:approximation_problem} has a positive answer by the main
  result of L\"uck's article~\cite{Lueck(1994c)}.

  \begin{theorem}\label{the:approx_Betti_char_zero}
    Let $F$ be a field of characteristic zero, and let $X$ be a finite   $CW$-complex. Then for each finite index normal chain $\{ G_i \}$ we have
    \[
    \lim_{i \to \infty} \frac{b_n(X[i];F)}{[G:G_i]} = b^{(2)}_n(\overline{X};\caln(G)),
    \]
    where $b^{(2)}(\overline{X};\caln(G))$ denotes the $n$th $L^2$-Betti number.
  \end{theorem}

  L\"oh and Uschold~\cite[Proposition~6.6]{Loeh-Uschold(2022)} prove  a quantitative version of
  Theorem~\ref{the:approx_Betti_char_zero}. Essentially they show
  \[
    \left| b^{(2)}_n(\overline{X};\caln(G)) -   \frac{b_n(X[i];F)}{[G:G_i]}\right|
    \le
    a \cdot \left(1 - \frac{1}{b\cdot i}\right)^{i^2} + \frac{a \cdot \log(b)}{\log(i)}
  \]
  for certain constants $a$ and $b$ depending on $X$.


  \subsection{Approximation of $L^2$-Betti numbers in prime characteristic}%
  \label{subsec:Approximation_of_L2_-Betti_numbers_in_prime_characteristic}

  The situation is more complicated and unclear in prime characteristic.
  Fix a prime $p$. Let $F$ be a field of characteristic $p$. We consider the
  \emph{$n$th Betti number with $F$-coefficients} $b_n(X;F) := \dim_F(H_n(X;F))$.  Note
  that $b_n(X;F) = b_n(X;\IF_p)$ holds where $\IF_p$ is the field of $p$-elements.  In this
  setting a general answer to Problem~\ref{pro:approximation_problem} is only known in
  special cases.  The main problem is that one does not have an analogue of the von Neumann
  algebra in characteristic $p$ and the construction of an appropriate extended dimension
  function, see~\cite{Lueck(1998a)}, is not known in general.

  If $G$ is torsionfree elementary amenable, one gets the full positive answer
  by Linnell-L\"uck-Sauer~\cite[Theorem~0.2]{Linnell-Lueck-Sauer(2011)}, where more
  explanations, e.g., about Ore localizations, are given and actually virtually torsionfree elementary
  amenable groups are considered.

  \begin{theorem}\label{the:dim_approximation_over_fields}
    Let $F$ be a field (of arbitrary characteristic) and $X$ be a connected finite
    $CW$-complex.  Let $G$ be a torsionfree elementary amenable group. Then for each finite index normal chain $\{ G_i \}$:
    \[
    \dim_{FG}^{\Ore} \bigl(H_n(\overline{X};F)\bigr) = \lim_{n \to \infty}
    \frac{b_n(X[i];F)}{[G:G_n]}.
    \]
  \end{theorem}

  Note that Theorem~\ref{the:dim_approximation_over_fields} is consistent with
  Theorem~\ref{the:approx_Betti_char_zero} since for a field $F$ of characteristic zero
  and  a torsionfree elementary amenable group $G$ we
  have $b_n^{(2)}(\overline{X};\caln(G)) = \dim_{FG}^{\Ore}
  \bigl(H_n(\overline{X};F)\bigr)$. The latter equality follows
  from~\cite[Theorem~6.37 on page~259, Theorem~8.29 on page~330,
  Lemma~10.16 on page~376, and Lemma~10.39 on page~388]{Lueck(2002)}.

  Here is another special case taken from
  Bergeron-L\"uck-Linnell-Sauer~\cite{Bergeron-Linnell-Lueck-Sauer(2014)},
  (see also Calegari-Emerton~\cite{Calegari-Emerton(2009bounds), Calegari-Emerton(2011)}), where we know
  the answer only for special chains.  Let $p$ be a prime, let $n$ be a positive integer,
  and let $\phi\colon G \rightarrow \GL_n (\IZ_p)$ be a homomorphism, where $\IZ_p$ denotes
  the $p$-adic integers. The closure of the image of $\phi$, which is denoted by $\Lambda$,
  is a $p$-adic analytic group admitting an exhausting filtration by open normal subgroups
  $\Lambda_i = \ker \left( \Lambda \rightarrow \GL_n (\IZ / p^i \IZ) \right)$. Put $G_i = \phi^{-1} (\Lambda_i )$.

  \begin{theorem}\label{the:BLLS}
    Let $F$ be a field (of arbitrary characteristic). Put $d= \dim  (\Lambda)$.
    Let $X$ be a finite $CW$-complex.  Then for any integer $n$ and
    as $i$ tends to infinity, we have:
    \[
    b_n(X[i];F) = b_n^{(2)}(\overline{X};F) \cdot [G : G_i] + O\left([G :
      G_i]^{1-{1/d}} \right),
    \]
    where $b_n^{(2)}(\overline{X};F)$ is the $n$th mod $p$ $L^2$-Betti numbers occurring
    in~\cite[Definition~1.3]{Bergeron-Linnell-Lueck-Sauer(2014)} which is defined using
    homology coefficients in the Iwasawa algebra of $G$.  In particular
    \[
    \lim_{i \to \infty} \frac{b_n(X[i] ;F)}{[G:G_i]} = b_n^{(2)}(\overline{X};F).
    \]
  \end{theorem}

  Returning to the setting of arbitrary finite index normal chain $(G_i)_{i \ge 0}$, we get by the universal
  coefficient theorem $b_n(X[i];\IQ) \le b_n(X[i];F)$ for any field $F$ and hence by
  Theorem~\ref{the:approx_Betti_char_zero} the inequality
  \[
    \liminf_{i \to \infty} \frac{b_n(X[i] ;F)}{[G:G_i]} \ge
    b_n^{(2)}(\overline{X};\caln(G)).
  \]
  If $p$ is a prime and we additionally assume that each index $[G:G_i]$ is a $p$-power,
  then the sequence $\frac{b_n(X[i] ;F)}{[G:G_i]}$ is monotone decreasing and in particular
  $\lim_{i \to \infty} \frac{b_n(X[i] ;F)}{[G:G_i]}$ exists,
  see~\cite[Theorem~1.6]{Bergeron-Linnell-Lueck-Sauer(2014)}.

  \begin{question}[Approximation in prime characteristic]\label{que:Approximation_in_zero_and_prime_characteristic}
    Does the sequence $\frac{b_n(X[i] ;F)}{[G:G_i]}$ converge and is the limit
    $\lim_{i \to \infty} \frac{b_n(X[i] ;F)}{[G:G_i]}$ independent of the chain $\{G_i\}$ for all
    fields $F$ and $n \ge 0$, provided that $X$ is finite and $\overline{X}$ is
    contractible?
  \end{question}

  \begin{remark}\label{rem:not_the_L2-Betti_number}
    The third author conjectured in the situation of
    Question~\ref{que:Approximation_in_zero_and_prime_characteristic} that the
    $\lim_{i \to \infty} \frac{b_n(X[i] ;F)}{[G:G_i]}$ is equal to
    $b_n^{(2)}(X;\caln(G))$ for every field $F$ and for all finite index normal chains $\{G_i\}$, see~\cite[Conjecture~3.4 on
    page~275]{Lueck(2016_l2approx)}.  This is true in characteristic zero by
    Theorem~\ref{the:approx_Betti_char_zero} and if $G$ is torsionfree elementary amenable by
    Theorem~\ref{the:dim_approximation_over_fields}, but not in prime characteristic by
    Avramidi-Okun-Schreve~\cite[Corollary~2]{Avramidi-Okun-Schreve(2021)}.  Namely,
    Avramidi-Okun-Schreve~\cite[Theorem~1]{Avramidi-Okun-Schreve(2021)} prove that for a
    right-angled Artin group $A_L$ with defining flag complex $L$ and any field $F$ the limit
    $\lim_{i \to \infty} \frac{b_n(X[i] ;F)}{[G:G_i]}$ exists, is independent of the
    chain $\{G_i\}$, and actually agrees with the reduced Betti number $\overline{b}_n(L;F)$ of $L$
    with coefficients in $F$. Furthermore, they construct examples where
    $\overline{b}_3(L;\IQ) = 0$ and $\overline{b}_3(L;\IF_2) = 1$ hold. These counterexamples do not
    exist in degree $n =1$.
  \end{remark}

  \begin{remark}[$L^2$-Betti numbers in finite characteristic via skew
    fields]\label{rem:embedding_group_ring_skew_field}
    Recall from Theorem~\ref{the:Main_properties_of_cald(G)} that if a torsionfree group $G$
    satisfies the Atiyah conjecture, then the group ring $\IC G$ embeds into a skew field
    $\cald(G)$ and $L^2$-Betti numbers are the Betti numbers with coefficients in this
    skew field.  Motivated by this, one might hope that for a torsionfree group $G$ and
    any field $F$ there exists a skew field $\cald_F(G)$ together with an embedding
    $F G \hookrightarrow \cald_F(G)$ which can be used to define $L^2$-Betti numbers with
    coefficients in $F$ as
    $b^{(2)}_n(\overline{X}; F) \coloneqq \dim_{\cald_F(G)}H_n^G(\overline{X};
    \cald_F(G))$.  An approach to Question~\ref{que:Approximation_in_zero_and_prime_characteristic} would then be to show that
    the sequence $\frac{b_n(X[i] ;F)}{[G:G_i]}$ converges to $b^{(2)}_n(\overline{X}; F)$
    for $i \to \infty$.

    Jaikin-Zapirain constructs such embeddings in~\cite[Corollary
    1.3]{Jaikin-Zapirain(2021)} for large classes of groups.  He shows that such an
    embedding of $FG$ into a skew field with very nice properties (it is a Hughes free
    division ring and the universal division ring of fractions of $FG$) exists if $G$ is
    residually (locally indicable amenable).  Combining the approximation results~\cite[Theorem 1.2]{Jaikin-Zapirain(2021)}
    and~\cite[Theorem~0.2]{Linnell-Lueck-Sauer(2011)} it is easy to check that in this
    situation the equality
    \begin{equation}
      b^{(2)}_n(\overline{X}; F) = \inf \left\{ \frac{b_n(\overline{X}/H; F)}{[G:H]} : H \le G, [G:H] < \infty \right\} \label{eq:approximation_skewfield}
    \end{equation}
    holds.  In the case $F = \IQ$ of the classical approximation theorem the easier to
    prove Kazhdan's inequality predicts that $L^2$-Betti numbers are bounded by
    \begin{equation*}
      b^{(2)}_n(\overline{X}; \caln(G)) \ge \limsup_{i \to \infty} \frac{b_n(X[i]; \IQ)}{[G:G_i]}
    \end{equation*}
    for any finite index normal chain $(G_i)$ for $G$.  It is not clear, however, whether
    an analogue of this holds in the present setting.  This would imply
    that~\eqref{eq:approximation_skewfield} still holds when one replaces $\inf$ by the limit
    over a finite index normal chain.
  \end{remark}


  \subsection{Torsion}%
  \label{subsec:Torsion}

  For a detailed discussion of approximation for finite index normal chains for torsion invariants
  such as the Ray-Singer torsion or the integral torsion we refer for
  instance to~\cite[Sections~8 -- 10]{Lueck(2016_l2approx)} and also to
  Subsection~\ref{subsec:Approximation_of_L2_-torsion}.


  \subsection{Homological torsion growth and $L^2$-torsion}%
  \label{subsec:Torsion_growth_and_L2-torsion}

  The following conjecture is taken from~\cite[Conjecture~1.12~(2)]{Lueck(2013l2approxfib)}.
  For locally symmetric spaces it reduces to the conjecture of Bergeron and
  Venkatesh~\cite[Conjecture~1.3]{Bergeron-Venkatesh(2013)}.

  \begin{conjecture}[Homological torsion growth and $L^2$-torsion]%
  \label{con:Homological_torsion_growth_and_L2-torsion}
  Let $M$ be an aspherical closed manifold.

     Then we get for  any natural number $n$ with $2n +1 \not= \dim(M)$
      \[
      \lim_{i \to \infty} \;\frac{\ln\big(\bigl|\tors\bigl(H_n(M[i];\IZ)\bigr)\bigr|\bigr)}{[G:G_i]} = 0.
      \]
      If the dimension $\dim(M) = 2m+1$ is odd, then $\widetilde{M}$ is $\det$-$L^2$-acyclic and we get
      \[
      \lim_{i \to \infty} \;\frac{\ln\big(\bigl|\tors\bigl(H_m(M[i];\IZ)\bigr)\bigr|\bigr)}{[G:G_i]} = (-1)^m \cdot
      \rho^{(2)}(\widetilde{M}).
      \]
    \end{conjecture}

    Conjecture~\ref{con:Homological_torsion_growth_and_L2-torsion} is true,
    if $\pi_1(M)$ contains an elementary amenable infinite normal  subgroup or $M$ carries a
    non-trivial $S^1$-action,   see~\cite[Corollary~1.13]{Lueck(2013l2approxfib)}.

    \begin{remark}[Criteria for vanishing homology torsion growth]
      Recently, there has been some progress in proving vanishing of homology growth and
      torsion homology growth for certain groups.  These results are based on the general
      vanishing result of Abert-Bergeron-Fraczyk-Gaboriau~\cite[Theorem
      10.20]{Abert-Bergeron-Fraczyk-Gaboriau(2021)} and have been applied to show
      vanishing for special linear groups and mapping class groups~\cite[Theorem A,
      D]{Abert-Bergeron-Fraczyk-Gaboriau(2021)}.  Further applications of this theorem can
      be found in~\cite[Theorem A]{Andrew-Guerch-Hughes-Kudlinska(2023)},~\cite[Theorem
      A]{Andrew-Hughes-Kudlinska(2022)}, and~\cite[Proposition 1.6]{Uschold(2022)}, where
      vanishing of homology growth and torsion homology growth of mapping tori of finitely
      generated free groups with respect to a polynomially growing automorphism and of
      right-angled Artin groups which are inner amenable is shown.  Combining the first
      result with~\cite[Theorem 5.1]{Clay(2017free)} confirms
      Conjecture~\ref{con:Homological_torsion_growth_and_L2-torsion} for mapping tori of
      aspherical closed manifolds which have finitely generated free fundamental group and
      the monodromy is polynomially growing.
    \end{remark}

    \begin{remark}[The Singer Conjecture~\ref{con:Singer_Conjecture} and
      Conjecture~\ref{con:Homological_torsion_growth_and_L2-torsion} are not compatible]%
  \label{rem:Singer_and_growth}
      The Singer Conjecture~\ref{con:Singer_Conjecture} and
      Conjecture~\ref{con:Homological_torsion_growth_and_L2-torsion} about homological
      torsion growth and $L^2$-torsion cannot both be true in general. Namely, if both are
      true, then the $\IF_p$-Singer Conjecture of
      Subsection~\ref{subsec:The_F_p-Singer_Conjecture} would be true as pointed out by
      Avramidi-Okun-Schreve before Theorem~4 appearing
      in~\cite{Avramidi-Okun-Schreve(2021)}. The argument is as follows.

      Suppose we have an aspherical closed $d$-dimensional manifold
      $M$ and both the Singer Conjecture~\ref{con:Singer_Conjecture} and
      Conjecture~\ref{con:Homological_torsion_growth_and_L2-torsion} are true.
      Consider $n$ with $2n \not= d$. We have to show
      \begin{eqnarray}
       \lim_{i \to  \infty} \frac{b_{n}(M[i];\IF_p)}{[G:G_i]} = 0
        \label{to_show_for_F_p-Singer}
      \end{eqnarray}
      for $G = \pi_1(M)$ and a finite index normal chain $\{G_i\}$.
      Because of Poincare duality it suffices to consider the case $2n > d$.
      Put $M^3= M \times M \times M$, $G^3 = G \times G \times G$ and $G_i^3 = G_i \times G_i \times G_i$.
      Then $\widetilde{M^3} = \widetilde{M} \times \widetilde{M}\times \widetilde{M}$, $G^3 = \pi_1(M^3)$,
      and $M^3[i] = M[i]^3$.
     By the Singer Conjecture~\ref{con:Singer_Conjecture} applied to $M^3$ we get
     $b_{3n}^{(2)}(\widetilde{M^3}) = 0$.
     This implies by Theorem~\ref{the:dim_approximation_over_fields}
     \[
      \lim_{i \to  \infty}
      \frac{\rk_{\IZ}(H_{3n}(M[i]^3;\IZ))}{[G^3:G_i^3]}  = \lim_{i \to  \infty} \frac{b_{3n}(M[i]^3;\IQ)}{[G^3:G_i^3]}
      = b_{3n}^{(2)}(\widetilde{M^3};\IQ) = 0.
     \]
     Since $2(3n) > 2(3n-1) > 3d$ holds, we conclude from Conjecture~\ref{con:Homological_torsion_growth_and_L2-torsion}
     \begin{eqnarray*}
       \lim_{i \to \infty} \;\frac{\ln\big(\bigl|\tors\bigl(H_{3n}(M^3[i];\IZ)\bigr)\bigr|\bigr)}{[G^3:G_i^3]} & = & 0,
       \\
       \lim_{i \to \infty} \;\frac{\ln\big(\bigl|\tors\bigl(H_{3n-1}(M^3[i];\IZ)\bigr)\bigr|\bigr)}{[G^3:G_i^3]} & = & 0.
     \end{eqnarray*}
     We have for any $k \ge 0$
     \begin{eqnarray*}
       \dim_{\IF_p}(\IZ^k \otimes_{\IZ} \IF_p) & =  & \dim_{\IZ}(\IZ^k);
        \\
       \dim_{\IF_p}(\IZ/p^k \otimes_{\IZ} \IF_p) & =  & 1;
       \\
       \dim_{\IF_p}\bigl(\Tor^{\IZ}_1(\IZ/p^k,\IF_p)\bigr) & =  & 1;
       \\
       \ln\bigl(|\tors(\IZ/p^k )|\bigr)  & =  & k \cdot \ln(p).
     \end{eqnarray*}
     If $q$ is an prime different from $p$, we get $\IZ/q^k \otimes_{\IZ} \IF_p = 0$
     and $\Tor^{\IZ}_1(\IZ/q^k,\IF_p) = 0$. Hence we get for any finitely generated abelian group $M$
     \begin{eqnarray*}
       \dim_{\IF_p}(M\otimes_{\IZ} \IF_p)
       & \le &
       \dim_{\IZ}(M) + \frac{\ln\bigl(|\tors(M)|\bigr)}{\ln(2)};
       \\
       \dim_{\IF_p}(\Tor^{\IZ}_1(M,\IF_p)\bigr)
       & \le &
       \frac{\ln\bigl(|\tors(M)|\bigr)}{\ln(2)}.
     \end{eqnarray*}
     We conclude from the K\"unneth Formula and the Universal Coefficient Theorem
     \begin{eqnarray*}
        \frac{b_n(M[i];\IF_p)^3}{[G:G_i]^3}
       & = &
       \frac{b_n(M[i];\IF_p)^3}{[G^3:G_i^3]}
       \\
       & \le &
       \frac{b_{3n}(M[i]^3;\IF_p)}{[G^3:G_i^3]}
       \\
        & = &
              \frac{\dim_{\IF_p}\bigl(H_{3n}(M[i]^3;\IZ) \otimes_{\IZ} \IF_p\bigr) +
              \dim_{\IF_p}\bigl(\Tor^{\IZ}_1(H_{3n-1}(M[i];\IZ),\IF_p)\bigr)}{[G^3:G_i^3]}
       \\
        & \le &
        \frac{\dim_{\IZ}(H_{3n}(M[i]^3;\IZ))}{[G^3:G_i^3]}
           + \frac{\ln\big(\bigl|\tors\bigl(H_{3n}(M^3[i];\IZ)\bigr)\bigr|\bigr)}{[G^3:G_i^3] \cdot \ln(2)} \\
           && \hspace{30mm} + \frac{\ln\big(\bigl|\tors\bigl(H_{3n-1}(M^3[i];\IZ)\bigr)\bigr|\bigr)}{[G^3:G_i^3] \cdot \ln(2)}.
     \end{eqnarray*}
     Hence~\eqref{to_show_for_F_p-Singer} is true.

      We have already mentioned in
      Subsection~\ref{subsec:The_F_p-Singer_Conjecture} that the $\IF_p$-Singer Conjecture
      is not true in general, see~\cite[Theorem~4]{Avramidi-Okun-Schreve(2021)}.
      Hence the Singer Conjecture~\ref{con:Singer_Conjecture} or
      Conjecture~\ref{con:Homological_torsion_growth_and_L2-torsion} is not true.
    \end{remark}

    In view of Remark~\ref{rem:Singer_and_growth} we suggest to keep the Singer Conjecture~\ref{con:Singer_Conjecture}
    and weaken  Conjecture~\ref{con:Homological_torsion_growth_and_L2-torsion} to the following version.

    \begin{conjecture}[Modified Homological torsion growth and $L^2$-torsion]%
      \label{con:Modified_Homological_torsion_growth_and_L2-torsion}
      Let $M$ be an aspherical closed manifold of odd dimension $\dim(M) = 2m+1$.

      Then $M$ is $\det$-$L^2$-acyclic,
      the limit
      $\lim_{i \to \infty} \left(\sum_{n = 0}^{2m+1} (-1)^n\cdot
        \frac{\ln(|\tors(H_n(M[i];\IZ))|)}{[G:G_i]}\right)$
      exists and is given by
      \[
        \lim_{i \to \infty} \left(\sum_{n = 0}^{2m+1} (-1)^n \cdot
          \frac{\ln\big(\bigl|\tors\bigl(H_n(M[i];\IZ)\bigr)\bigr|\bigr)}{[G:G_i]}\right) =
       \rho^{(2)}(\widetilde{M}).
      \]
    \end{conjecture}

    \begin{remark}[Discussion of Conjecture~\ref{con:Modified_Homological_torsion_growth_and_L2-torsion}]%
  \label{rem:motivation_Modifiied_Conjecture}
      It is conceivable that both the Singer Conjecture~\ref{con:Singer_Conjecture} and the
      modified
      Conjecture~\ref{con:Modified_Homological_torsion_growth_and_L2-torsion}
      about homological torsion growth and $L^2$-torsion are true,
      since the argument in Remark~\ref{rem:Singer_and_growth} that the Singer
      Conjecture~\ref{con:Singer_Conjecture} or
      Conjecture~\ref{con:Homological_torsion_growth_and_L2-torsion} is not true does not
      apply anymore.

      The difference between Conjecture~\ref{con:Homological_torsion_growth_and_L2-torsion}
      and Conjecture~\ref{con:Modified_Homological_torsion_growth_and_L2-torsion} is that we
      drop the claim $\lim_{i \to \infty} \;\frac{\ln(|\tors(H_n(M[i];\IZ))|)}{[G:G_i]} = 0$
      in Conjecture~\ref{con:Homological_torsion_growth_and_L2-torsion} and allow in
      Conjecture~\ref{con:Modified_Homological_torsion_growth_and_L2-torsion} contributions
      from all $n \ge 0$ and not only from $n = m$ as in
      Conjecture~\ref{con:Homological_torsion_growth_and_L2-torsion}.  We conclude
      $\tors\bigl(H_n(M[i];\IZ)\bigr) \cong \tors\bigl(H_{2m-n}(M[i];\IZ)\bigr)$ for
      $n \ge 0$ from Poincar\'e duality and the universal coefficient theorem and the vanishing of
      $\tors\bigl(H_0(M[i];\IZ)\bigr)$.  We get the equality
      \begin{multline*}
      \sum_{n = 0}^{2m+1} (-1)^n\cdot
      \frac{\ln\big(\bigl|\tors\bigl(H_n(M[i];\IZ)\bigr)\bigr|\bigr)}{[G:G_i]}
      \\
      =
      (-1)^m \cdot \frac{\ln\big(\bigl|\tors\bigl(H_m(M[i];\IZ)\bigr)\bigr|\bigr)}{[G:G_i]} +
      2 \cdot \sum_{n = 1}^{m-1} (-1)^n\cdot
      \frac{\ln\big(\bigl|\tors\bigl(H_n(M[i];\IZ)\bigr)\bigr|\bigr)}{[G:G_i]}.
    \end{multline*}
    Hence in dimension $\dim(M)  = 3$ Conjecture~\ref{con:Homological_torsion_growth_and_L2-torsion}
      and Conjecture~\ref{con:Modified_Homological_torsion_growth_and_L2-torsion} agree.
      This does not lead to any contradiction since no $3$-dimensional counterexample to the
      $\IF_p$-Singer Conjecture known.

      Note that Conjecture~\ref{con:Modified_Homological_torsion_growth_and_L2-torsion} does
      not imply the Hopf Conjecture for
      $L^2$-torsion~\ref{con:Hopf_Conjecture_for_L2-torsion} for aspherical closed
      topological manifolds. One may wonder whether the Hopf Conjecture for
      $L^2$-torsion~\ref{con:Hopf_Conjecture_for_L2-torsion} is not true in odd
      dimensions $\dim(M) \ge 5$  in general.
      It is known to be true in dimension $\dim(M) = 3$
      by~\cite[Theorem~0.7]{Lueck-Schick(1999)}, yet again since Thurston's Geometrization
      Conjecture is known to be true.
      The modified Conjecture~\ref{con:Modified_Homological_torsion_growth_and_L2-torsion} resolves the problem mentioned above, but may only be true in special cases.

      One may replace in
      Conjecture~\ref{con:Modified_Homological_torsion_growth_and_L2-torsion} the aspherical
      closed manifold $M$ of odd dimension $\dim(M) = 2m+1$ by a connected finite
      $CW$-complex $X$, for which $\widetilde{X}$ is $\det$-$L^2$-acyclic.  This
      corresponds to~\cite[Conjecture~8.9 on page~290]{Lueck(2016_l2approx)}. This version is true if
      $\pi_1(X) \cong \IZ$ holds, see~\cite[Theorem~7.3]{Bergeron-Venkatesh(2013)}, but at the time of writing
      nothing is known in general if $\pi_1(X)$ is infinite and not equal to $\IZ$.

    \end{remark}

    For applications of these approximation conjectures for $L^2$-torsion to questions about profinite rigidity
    for fundamental groups of $3$-manifolds and lattices in higher rank Lie  groups
    we refer to~\cite[Section~6.7]{Kammeyer(2019)} and~\cite{Kammeyer-Kionke(2023)}.
    Estimates for the homological growth in terms of the volume for aspherical manifolds are established
    in~\cite{Sauer(2016)}.


  \subsection{Further examples}\label{subsec:further_examples}

  Other invariants, such as the rank gradient and the cost, truncated Euler characteristics,
  and minimal numbers of generators, are discussed in~\cite[Sections~4 and~5]{Lueck(2016_l2approx)}.
  For consideration concerning the speed of convergence we refer to~\cite[Section~6]{Lueck(2016_l2approx)}.


  \typeout{------------------------ Section 7: Approximation  for arbitrary normal chains  ------------------}

  \section{Approximation for normal chains}%
  \label{sec:Approximation_for__normal_chains}

  A \emph{normal chain} $\{G_i\}$ for the group $G$ is a descending chain of subgroups
  \begin{eqnarray}
    & G = G_0 \supseteq G_1 \supseteq G_2 \supseteq \cdots &
                                                             \label{normal_chain}
  \end{eqnarray}
  such that $G_i$ is normal in $G$ and $\bigcap_{i \ge 0} G_i = \{1\}$. Note that a normal
  chain is a finite index normal chain, if and only if  $[G:G_i]$ is finite
  for each $i$. Next we want to discuss approximation results for normal chains.


  \subsection{The Approximation Conjecture  for $L^2$-Betti numbers}%
  \label{subsec:The_Approximation_Conjecture_for_L2-Betti_numbers}

  Next we deal with the Approximation Conjecture for $L^2$-Betti numbers
  (see~\cite[Conjecture~1.10]{Schick(2001b)},~\cite[Conjecture~13.1 on page~453]{Lueck(2002)}).

  \begin{conjecture}[Approximation Conjecture for $L^2$-Betti numbers]%
  \label{con:Approximation_conjecture_for_L2-Betti_numbers}
    A group $G$   satisfies the \emph{Approximation
      Conjecture for $L^2$-Betti numbers} if for any  normal chain $\{G_i\}$ one of the following equivalent
    conditions holds

    \begin{enumerate}

    \item Matrix version\\[1mm]
      Let $A \in M_{r,s}(\IQ G)$ be a matrix. Then
      \begin{eqnarray*}
        \lefteqn{\dim_{\caln(G)}\bigl(\ker\bigl(r_A^{(2)}\colon L^2(G)^r \to L^2(G)^s
          \bigr)\bigr)}
        & &
        \\ & \hspace{14mm} =  &
        \lim_{i \to \infty} \;\dim_{\caln(G/G_i)}\big(\ker
        \big(r_{A[i]}^{(2)}\colon L^2(G/G_i)^r \to L^2(G/G_i)^s \bigr)\bigr);
      \end{eqnarray*}

    \item $CW$-complex version\\[1mm]
      Let $X$ be a $G$-$CW$-complex of finite type. Then $X[i] := G_i\backslash X$ is a
      $G/G_i$-$CW$-complex of finite type and we get for $n \ge 0$
      \begin{eqnarray*} b_n^{(2)}(X;\caln(G)) & = & \lim_{i \to \infty}
        \;b_n^{(2)}(X[i];\caln(G/G_i)).
      \end{eqnarray*}

    \end{enumerate}
  \end{conjecture}

  The two conditions appearing in
  Conjecture~\ref{con:Approximation_conjecture_for_L2-Betti_numbers} are
  equivalent by~\cite[Lemma 13.4 on page~455]{Lueck(2002)}.

  \begin{remark}[The Determinant Conjecture  and the Approximation Conjecture]%
 \label{rem:The_Determinant_Conjecture_implies_the_Approximation_Conjecture_for_L2-Betti_numbers}
    Let $G$ be a group and $\{G_i\}$ be a normal chain. Suppose that $G$ and each $G/G_i$
    satisfies the Determinant Conjecture~\ref{con:Determinant_Conjecture}.  Then the
    Approximation Conjecture~\ref{con:Approximation_conjecture_for_L2-Betti_numbers} for
    $L^2$-Betti numbers holds for $G$ for this normal chain $\{G_i\}$
    by~\cite[Theorem~13.3 (1) on page~454]{Lueck(2002)}. We mention
    that in~\cite[Theorem~13.3 (1) on page~454]{Lueck(2002)} there is a misprint, $G_i$ has
    to be replaced by $G/G_i$. Moreover, as pointed out to us by
    Bin Sun,~\cite[Theorem~13.3 (2) on page~454]{Lueck(2002)} is not correct as stated.
    It is true that $G$ is of $\det \ge 1$-class if the group $G$ belongs to $\calg$. But it is not
    correct that $G$ satisfies the Approximation Conjecture for any normal chain $\{G_i\}$
    if $G$ belongs to $\calg$, one additionally needs that each $G/G_i$ belongs to $\calg$.
  \end{remark}

  Recall that the Determinant Conjecture~\ref{con:Determinant_Conjecture} is known for a large class of groups,
  for instance it is true for all sofic groups.
  Suppose that each quotient $G/G_i$ is finite. Then we recover
  Theorem~\ref{the:approx_Betti_char_zero} from
  Remark~\ref{rem:The_Determinant_Conjecture_implies_the_Approximation_Conjecture_for_L2-Betti_numbers}.

  A typical application of the Approximation
  Conjecture~\ref{con:Approximation_conjecture_for_L2-Betti_numbers} is the following. If
  $G$ satisfies the Approximation
  Conjecture~\ref{con:Approximation_conjecture_for_L2-Betti_numbers} and
  for a given normal chain $\{G_i\}$ each $G/G_i$ is torsionfree and satisfies the Atiyah
  Conjecture~\ref{con:Atiyah_Conjecture}, then $G$ is torsionfree and satisfies the Atiyah
  Conjecture~\ref{con:Atiyah_Conjecture}, since a limit of a convergent sequence of integers
  is an integer again.

  For more information about the Approximation Conjecture and its applications we refer
  to~\cite[Chapter~13]{Lueck(2002)} and~\cite{Schick(2001b)}.


  \subsection{Approximation of Fuglede-Kadison determinants}%
  \label{subsec:Approximation_of_Fuglede-Kadison_determinants}

  The following conjecture is taken from~\cite[Conjecture~14.1 on page~308]{Lueck(2016_l2approx)}.

  \begin{conjecture}[Approximation Conjecture for Fuglede-Kadison determinants]%
  \label{con:Approximation_conjecture_for_Fuglede-Kadison_determinants_with_arbitrary_index}
  A group $G$ satisfies
  the \emph{Approximation Conjecture for Fuglede-Kadison determinants}
  if for any normal chain $\{G_i\}$ and any matrix $A \in M_{r,s}(\IQ G)$
  we get for the Fuglede-Kadison determinant
  \begin{eqnarray*}
  \lefteqn{{\det}_{\caln(G)}\bigl(r_A^{(2)}\colon L^2(G)^r \to L^2(G)^s\bigr)}
  & &
  \\ & \hspace{14mm} =  &
  \lim_{i \in I}\; {\det}_{\caln(G/G_i)}\big(r_{A[i]}^{(2)}\colon L^2(G/G_i)^r \to L^2(G/G_i)^s\bigr).
  \end{eqnarray*}
  \end{conjecture}

  An equivalent chain complex version can be found in~\cite[Conjecture~15.3 on page~309]{Lueck(2016_l2approx)}.


  \subsection{Approximation of $L^2$-torsion}\label{subsec:Approximation_of_L2_-torsion}

  Let $\overline{M}$ be a Riemannian manifold without boundary that
  comes with a proper free cocompact isometric $G$-action.
  Denote by $M[i]$ the Riemannian manifold obtained
  from $\overline{M}$ by dividing out the $G_i$-action. The Riemannian metric on
  $M[i]$ is induced by the one on $M$. There is an obvious proper free
  cocompact isometric $G/G_i$-action on $M[i]$ induced by the given $G$-action on
  $\overline{M}$. Notice that $M = \overline{M}/G$ is a closed Riemannian manifold
  and we get a $G$-covering $\overline{M} \to M$ and a $G/G_i$-covering $M[i] \to M$
  which are compatible with the Riemannian metrics. Denote by
  \begin{eqnarray}
    \rho^{(2)}_{\an}(\overline{M};\caln(G)) & \in & \IR;
    \label{L2-torsion_for_M_over_cakln(G)}
    \\
    \rho^{(2)}_{\an}(M[i];\caln(G/G_i)) & \in & \IR,
    \label{L2-torsion_for_M[i]_over_caln(G/G_i)}
  \end{eqnarray}
  their \emph{analytic $L^2$-torsion} over $\caln(G)$ and $\caln(G/G_i)$
  respectively, see~\cite[Definition~3.128 on page~178]{Lueck(2002)}. Note  that we do \emph{not} require
  that $\overline{M}$ or $M[i]$ is $L^2$-acyclic. If $[G:G_i]$ is finite,
  $\rho^{(2)}_{\an}(M[i];\caln(G/G_i))$ is the Ray-Singer torsion of the closed Riemannian manifold
  multiplied with $\frac{1}{[G:G_i]}$ defined in~\cite{Ray-Singer(1971)}.

  \begin{conjecture}[Approximation Conjecture for analytic $L^2$-torsion]%
  \label{con:Approximation_conjecture_for_analytic_L2-torsion}
  A group $G$ satisfies the \emph{Approximation Conjecture for analytic $L^2$-torsion}
  if for any normal chain $\{G_i\}$ and Riemannian manifold  $\overline{M}$ without boundary and with a proper free cocompact isometric $G$-action
  \[\rho^{(2)}_{\an}(\overline{M};\caln(G))
  = \lim_{i \in I} \;\rho^{(2)}_{\an}(M[i];\caln(G/G_i)).
  \]
  \end{conjecture}

  There are topological counterparts which we will denote by $\rho^{(2)}_{\topo}(X[i];\caln(G/G_i))$ and
  $\rho^{(2)}_{\topo}(\overline{X};\caln(G))$ see~\cite[Definition~3.120 on page~176]{Lueck(2002)}.
  They agree with their analytic versions
  by~\cite{Burghelea-Friedlander-Kappeler-McDonald(1996a)}.  So
  Conjecture~\ref{con:Approximation_conjecture_for_analytic_L2-torsion} is equivalent to its
  topological counterpart.

  \begin{conjecture}[Approximation Conjecture for topological torsion]\label{con:Approximation_for_topological_torsion}
    A group $G$ satisfies the \emph{Approximation Conjecture for topological $L^2$-torsion}
  if for any normal chain $\{G_i\}$ and Riemannian manifold  $\overline{M}$ without boundary and with a proper free cocompact isometric $G$-action
    \[
    \rho^{(2)}_{\topo}(\overline{M};\caln(G)) = \lim_{i \to \infty}
    \rho_{\topo}(M[i];\caln(G/G_I))
    \]
  \end{conjecture}

  The next result is proved in~\cite[Theorem~15.6 on page~310]{Lueck(2016_l2approx)}.

  \begin{theorem}\label{the:comparing_analytic_and_chain_complexes}
    Suppose that $G$ satisfies the
    Approximation Conjecture~\ref{con:Approximation_conjecture_for_Fuglede-Kadison_determinants_with_arbitrary_index}
    for Fuglede-Kadison determinants.

    Let $\overline{M}$ be a Riemannian manifold without
    boundary that comes with a proper free cocompact isometric $G$-action. Suppose
    that $b_n^{(2)}(\overline{M};\caln(G)) = 0$ holds for all $n \ge 0$.

    Then
    \[
    \rho^{(2)}_{\an}(\overline{M};\caln(G)) = \lim_{i \in I}
    \;\rho^{(2)}_{\an}(M[i];\caln(G/G_i)).
    \]
  \end{theorem}

  \begin{remark}[Relating the Approximation Conjectures for Fuglede-Kadison determinant and
    torsion invariants]\label{rem:Fugle_and_torsion}
  It is conceivable that Theorem~\ref{the:comparing_analytic_and_chain_complexes} is still true, if we drop the assumption that $b_n^{(2)}(\overline{M};\caln(G))$ vanishes for all $n \ge 0$,
  but our present proof works only under this assumption,
  see~\cite[Remark~16.2 on page~315]{Lueck(2016_l2approx)}. If we can drop this assumption
  in Theorem~\ref{the:comparing_analytic_and_chain_complexes}, then
  Theorem~\ref{the:comparing_analytic_and_chain_complexes} just boils down to the statement  that
  $G$ satisfies the equivalent Conjectures~\ref{con:Approximation_conjecture_for_analytic_L2-torsion}
  and~\ref{con:Approximation_for_topological_torsion} about $L^2$-torsion, provided that $G$ satisfies
  the Approximation Conjecture~\ref{con:Approximation_conjecture_for_Fuglede-Kadison_determinants_with_arbitrary_index}
  for Fuglede-Kadison determinants.
  \end{remark}

  \begin{remark}[Strategy of proof of
    Conjecture~\ref{con:Approximation_conjecture_for_Fuglede-Kadison_determinants_with_arbitrary_index}]%
    \label{rem:strategy_of_proof_of_appr_FK-det}
    A strategy to prove the Approximation
    Conjecture~\ref{con:Approximation_conjecture_for_Fuglede-Kadison_determinants_with_arbitrary_index}
    for Fuglede-Kadison determinants is described
    in~\cite[Theorem~17.1 on page~316]{Lueck(2016_l2approx)}.  The key problem is isolated in the
    \emph{uniform integrability condition} appearing
    in~\cite[Theorem~17.1~(v)]{Lueck(2016_l2approx)}. Why it does not follow from known
    facts about spectral density functions is explained
    in~\cite[Lemma~17.2 on page~318]{Lueck(2016_l2approx)}, whereas
    in~\cite[Remark~17.3 on page~321]{Lueck(2016_l2approx)} it is discussed why there is some hope that
    it may hold in many situations. In~\cite[Theorem~17.6 on page~322]{Lueck(2016_l2approx)} a
    sufficient condition for the validity of the uniform integrability condition is
    described.
  \end{remark}

  \begin{remark}[(Very poor) status of
    Conjecture~\ref{con:Approximation_conjecture_for_Fuglede-Kadison_determinants_with_arbitrary_index}]%
    \label{rem:status_of_appr_FK-det}
    Unfortunately, there is no infinite group $G$ besides infinite virtually cyclic groups
    for which the Approximation
    Conjecture~\ref{con:Approximation_conjecture_for_Fuglede-Kadison_determinants_with_arbitrary_index}
    for Fuglede-Kadison determinants is known to be true.  It holds for finite groups $G$
    for trivial reasons.  For $G = \IZ$ (and hence for every infinite virtually cyclic group)
    Conjecture~\ref{con:Approximation_conjecture_for_Fuglede-Kadison_determinants_with_arbitrary_index}
    has been proven by Schmidt~\cite{Schmidt(1995)}, see also~\cite[Lemma~13.53 on
  page~478]{Lueck(2002)} and~\cite[Lemma~7.25]{Lueck(2018)}.  Note that the proof of this
    very special case already requires some input from the theory of diophantine equations, namely the
    estimate~\cite[13.65]{Lueck(2002)} taken from~\cite[Corollary~B1 on page~30]{Shorey-Tijdeman(1986)}.

    We conclude from~\cite[Theorem~15.7]{Lueck(2016_l2approx)}
    that in the situation of the Approximation Conjecture for
    Fuglede-Kadison determinants~\ref{con:Approximation_conjecture_for_Fuglede-Kadison_determinants_with_arbitrary_index}
    we always have the inequality
    \begin{eqnarray}\label{inequality_for_Det_Appr}
         \lefteqn{\dim_{\caln(G)}\bigl(\ker\bigl(r_A^{(2)}\colon L^2(G)^r \to L^2(G)^s
          \bigr)\bigr)}
        & &
        \\ & \hspace{7mm} \ge   &
        \limsup_{i \to \infty} \;\dim_{\caln(G/G_i)}\big(\ker
                                \big(r_{A[i]}^{(2)}\colon L^2(G/G_i)^r \to L^2(G/G_i)^s \bigr)\bigr)
           \nonumber
    \end{eqnarray}
    provided  that each quotient $G/G_i$ satisfies the Determinant Conjecture~\ref{con:Determinant_Conjecture}.
    If $G = \IZ^n$ and each $G/G_i$ is finite, then the inequality~\eqref{inequality_for_Det_Appr}  for the limit superior
    is known to be an equality by
  L\^e~\cite[Theorem~3]{Le(2014Mahler)}, see also Raimbault~\cite{Raimbault(2012abelian)}.
\end{remark}

  \begin{remark}The relationship between Conjecture~\ref{con:Modified_Homological_torsion_growth_and_L2-torsion}
    about homological torsion growth and $L^2$-torsion
    and the Approximation Conjecture~\ref{con:Approximation_for_topological_torsion}  for finite index normal chains
   is described in terms of regulators
   in~\cite[Sections~8 -- 10]{Lueck(2016_l2approx)}.
  \end{remark}


  \subsection{More general setups}\label{subsec:More_general_setups}
  Some of the approximation conjectures above can also be formulated in
  more generality or a slightly different context. One can replace the normal chain $\{G_i\}$ indexed
  by $i \in \{0,1,2, \ldots \}$ by an inverse system
  $\{G_i \mid i \in I\}$ of normal subgroups of $G$ directed by
  inclusion over the directed set $I$ such that
  $\bigcap_{i \in I} G_i = \{1\}$,
  see~\cite[Section~13]{Lueck(2016_l2approx)}.  One may consider so
  called \emph{Farber sequences}, where the subgroups $G_i$ are not
  necessarily normal but normal in an asymptotic sense,
  see~\cite[Section~10]{Abert-Bergeron-Fraczyk-Gaboriau(2021)}. There is the notion of Benjamini-Schramm convergence,
  see~\cite[Definition~1.1]{Abert-Bergeron-Biringer-Gelander-Nikolov-Raimbault-Samet(2017)}.
  Or one can consider convergence in the space or marked groups,
  see~\cite[Section~1.3]{Jaikin-Zapirain+Lopez-Alvarez(2020)}.


  \typeout{------------------------ Section 8: $L^2$-invariants and simplicial volume -------------------}

  \section{$L^2$-invariants and simplicial volume}\label{sec:L2-invariants_and_simplicial_volume}

  The simplicial volume of a manifold is a topological variant of the (Riemannian) volume
  which agrees with it for hyperbolic manifolds up to a dimension constant and was introduced by Gromov~\cite{Gromov(1982)}.

  \begin{conjecture}[Simplicial volume and $L^2$-invariants]\label{con:simplicial_volume_and_L2-invariants}
  Let $M$ be an aspherical closed orientable manifold of dimension $\ge 1$.
  Suppose that its simplicial volume $\lVert M \rVert$ vanishes. Then
  $\widetilde{M}$ is of determinant class
  and
  \begin{eqnarray*}
  b_p^{(2)}(\widetilde{M}) & = & 0 \hspace{5mm} \mbox{ for } p \ge 0;
  \\
  \rho^{(2)}(\widetilde{M}) & = & 0.
  \end{eqnarray*}
  \end{conjecture}

  Gromov first asked in~\cite[Section~8A on page~232]{Gromov(1993)} whether under the
  conditions in Conjecture~\ref{con:simplicial_volume_and_L2-invariants} the Euler characteristic of
  $M$ vanishes, and notes that in all available examples even the $L^2$-Betti numbers of
  $M$ vanish.  The part about $L^2$-torsion appears
  in~\cite[Conjecture~3.2]{Lueck(1994a)}.

  Conjecture~\ref{con:simplicial_volume_and_L2-invariants} is discussed in detail
  in~\cite[Chapter~14]{Lueck(2002)}. No essential breakthrough has been made since then,
  but there are some interesting papers connected to this problem, such
  as~\cite{Frigerio-Loeh-Pagliantini_Sauer(2016), Sauer(2009amen), Sauer(2016)}.  So far,
  all evidence for Conjecture~\ref{con:simplicial_volume_and_L2-invariants} has been computational,
  but there is no structural reason known for it.  It is intriguing, since it relates
  rather different invariants to one another and is one of the typical conjectures, which
  only make sense for closed manifolds when they are aspherical.

  For an aspherical closed orientable manifold of dimension $\ge 1$ its simplicial volume $\lVert M \rVert$
  vanishes if $\pi_1(M)$ is amenable, see~\cite[page 40]{Gromov(1982)} and~\cite[Theorem
  4.3]{Ivanov(1987)}.  However, it is not known whether $\lVert M \rVert$ vanishes if $\pi_1(M)$
  contains a normal infinite amenable subgroup or at least an elementary amenable infinite
  normal subgroup.  Recall  that $\rho^{(2)}(\widetilde{M})$ vanishes if $\pi_1(M)$ contains a
  normal infinite elementary amenable subgroup,
  see~\eqref{list:main_properties_of_rho2(widetildeX):aspherical} in
  Subsection~\ref{subsec:Basic_properties_of_L2-Betti_numbers_and_L2-torsion}, but it is not
  known wether $\rho^{(2)}(\widetilde{M})$ vanishes if $\pi_1(M)$ contains a normal
  infinite amenable subgroup.

  As an approach to Conjecture~\ref{con:simplicial_volume_and_L2-invariants},
  Gromov suggested in~\cite{Gromov(1999a)} to use a new invariant, the integral foliated simplicial volume, to
  estimate $L^2$-Betti numbers from above.  This was carried out in~\cite[Corollary
  5.28]{Schmidt(2005)}.  However, it is still open whether vanishing of simplicial volume
  implies vanishing of integral foliated simplicial volume (of oriented closed aspherical
  manifolds).


  \typeout{------------------- Section 9: L^2-invariants of groups -------------------}

  \section{$L^2$-invariants of groups}\label{sec:L2-invariants_of_groups}

  Recall that $L^2$-Betti numbers $b_n^{(2)}(G)$  of a group $G$ were defined in~\eqref{b_n_upper_(2)(G)}.
  We call a group $G$ \emph{admissible}, if there exists a finite $CW$-model $BG$ for its
  classifying space, we have $b_n^{(2)}(EG;\caln(G)) = 0$ for $n \ge 0$, and $G$ satisfies the
  Determinant Conjecture~\ref{con:Determinant_Conjecture}. We define the \emph{$L^2$-torsion
    of an admissible group $G$}
  \begin{equation}
  \rho^{(2)}(G) =  \rho^{(2)}(EG;\caln(G)) \in \IR
  \label{rho_upper_(2)(G)}
  \end{equation}
  by the $L^2$-torsion of $EG = \widetilde{BG}$. Since two models for
  $EG$ are $G$-homotopy equivalent and $G$ satisfies the
  Determinant Conjecture~\ref{con:Determinant_Conjecture}, the real number
  $\rho^{(2)}(EG;\caln(G))$ is well-defined and is independent of the choice of
  the model for $EG$. Hence there notion of the $L^2$-torsion of an admissible group in~\eqref{rho_upper_(2)(G)}
  makes sense.


  \subsection{Vanishing results}\label{Vanishing_results}

  Some vanishing criterions for the $L^2$-Betti numbers $b_n^{(2)}(G)$ of a group can be
  found in~\cite[Theorem~7.2 on page 294 and Theorem~7.4 on page 295]{Lueck(2002)}. For
  instance, if $G$ contains a normal subgroup $H$ such that $b_n^{(2)}(H) = 0$ for all
  $n \le d$ for some fixed natural number $d$, then $b_n^{(2)}(G) = 0$ for all $n \le d$. If
  $G$ contains an amenable infinite normal subgroup, then $b_n^{(2)}(G) = 0$ holds for all
  $n \ge 0$.  If $G$ occurs as an extension $1 \to H \to G \to Q \to 1$ of infinite groups such that $H$
  is finitely generated or, more generally, satisfies $b_1^{(2)}(H) < \infty$, then $b_1^{(2)}(G) = 0$.
  See~\cite{Sanchez-Peralta(2023)} and~\cite{Sauer-Thom(2010)}
  for generalizations of the last assertion.  If $G$ is
  admissible and contains an elementary amenable infinite normal subgroup, then
  $\rho^{(2)}(G) = 0$ holds, see~\eqref{list:main_properties_of_rho2(widetildeX):aspherical}
  in Subsection~\ref{subsec:Basic_properties_of_L2-Betti_numbers_and_L2-torsion}. A very
  interesting interaction between the notions of $L^2$-torsion and entropy is developed
  in~\cite{Li-Thom(2014)} and used to show that for any admissible group which is amenable,
  $\rho^{(2)}(G) = 0$ holds, see~\cite[Theorem~1.3]{Li-Thom(2014)}. It is unknown whether
    $\rho^{(2)}(G) = 0$ vanishes if $G$ contains an amenable infinite normal
    subgroup.

    \begin{remark}\label{rem:combinatorial_computation}
      A combinatorial computation of $L^2$-invariants is described
      in~\cite[Section~3.7]{Lueck(2002)} Given a matrix $A \in M_{m,n}(\IC G)$, one can
      assign to it its \emph{characteristic sequence} $\{c(A,K)_n\}$ for some large
      enough number $K$.
      The sequence $\{c(A,K)_n\}$ is monotone decreasing sequence of
      non-negative real numbers that converges to the von Neumann dimension
      $\dim_{\caln(G)}(\ker(R_A))$ of the kernel of the induced $\caln(G)$-homomorphism
      $R_A \colon \caln(G)^m \to \caln(G)^n$. There is some control over the speed of
      convergence, which is often very fast, actually exponentially in $n$. Analogously there
      is a way of computing the Fuglede-Kadison determinant of $R_A$ in terms of the
      characteristic sequence. Given an algorithm to decide the word problem in $G$, one
      obtains an algorithm to compute the characteristic sequence. Unfortunately this
      algorithm seems to be exponentially running. But at least each $c(A,K)_n$ gives an
      upper bound for $\dim_{\caln(G)}(R_A)$.  This may be useful in view of the Atiyah
      Conjecture~\ref{con:Atiyah_Conjecture} to prove the vanishing of $L^2$-Betti numbers
      $b_n^{(2)}(\widetilde{X})$ for a finite $CW$-complex $X$ for which there is an upper
      bound on the orders of the finite subgroups of $\pi_1(X)$.  See also
      Remark~\ref{rem:Atiyah_without_bound_on_htr_orders_of_finite_subgroups}.
    \end{remark}


    \subsection{The first $L^2$-Betti number and applications to group theory}%
    \label{subsec:The_first_L2-Betti_numbers_and_applications_to_group_theory}%

    The vanishing of $b_1^{(2)}(G)$ has some interesting consequences, provided that $G$ is
    finitely presented.  Namely, it implies that the deficiency of $G$ is bounded from above
    by $1$ and that for any oriented closed manifold $M$ of dimension $M$ the inequality
    $|\sign(M)| \le \chi(M)$ holds for its signature $\sign(M)$ and its Euler characteristic
    $\chi(M)$, see~\cite[Theorems~5.1 and~6.1]{Lueck(1994b)}. The following result is due to
    Kielak~\cite[Theorem~5.3]{Kielak(2020fibring)} and generalizes the work of
    Agol~\cite{Agol(2008)}

      \begin{theorem}\label{the:First_L2_Betti_number_and_fibering}
        Let G be an infinite finitely generated group which is virtually \textup{RFRS},
        where \textup{RFRS} stands for residually finite rationally solvable.

        Then G is virtually fibered, in the sense that it admits a finite index subgroup
        mapping onto Z with a finitely generated kernel, if and only if $b^{(2)}_1(G) = 0$
        holds.
      \end{theorem}

      There are relations between the non-vanishing of $b_1^{(2)}(G)$ and the question
      whether the finitely presented group $G$ is \emph{large}, i.e, has a subgroup of
      finite index which maps surjectively to a non-abelian free group, see for
      instance~\cite[Theorem~1.4]{Lackenby(2009propery_tau)},~\cite[Theorem~1.6]{Lackenby(2010large)}.

      The following questions are taken from~\cite[Section~4]{Lueck(2016_l2approx)}, where
      also more explanations and references to the literature, such
      as~\cite{Abert-Jaikin-Zapirain-Nikolov(2011), Abert-Nikolov(2012),
        Bridson-Kochloukova(2017), Ershov-Lueck(2014), Gaboriau(2000b),
        Gaboriau(2002a),Gaboriau(2002b), Lackenby(2005expanders), Lueck-Osin(2011),
        Osin(2011_rankgradient), Schlage-Puchta(2012)}, are given.

      \begin{question}[Rank gradient, cost, and first $L^2$ Betti number]%
        \label{que:Rank_gradient_cost_and_first_L2_Betti_number}
        Let $G$ be an infinite finitely generated residually finite group. Let
        $(G_i)_{i \ge 0}$ be a descending chain of normal subgroups of finite index of $G$
        with $\bigcap_{i \ge 0} G_i=\{1\}$.

        Do we have
        \[
          b_1^{(2)}(G) = {\rm cost}(G)-1 = \RG(G;(G_i)_{i \ge 0})?
        \]
      \end{question}

      \begin{question}[Rank gradient, cost, first $L^2$-Betti number, and approximation]%
        \label{que:Rank_gradient_cost_first_L2_Betti_number_and_approximation}
        Let $G$ be a finitely presented infinite residually finite group. Let $(G_i)$ be a descending
        chain of normal subgroups of finite index of $G$ with
        $\bigcap_{i \ge 0 } G_i=\{1\}$.  Let $F$ be any field.

        Do we have
        \[
          \lim_{i \to \infty} \frac{b_1(G_i;F)-1}{[G:G_i]} = b_1^{(2)}(G) - b_0^{(2)}(G) =
          \costoper(G)-1 = \RG(G;(G_i)_{i \ge 0})?
        \]
      \end{question}

      Note that a positive answer to the questions above also includes the statement, that
      $\lim_{i \to \infty} \frac{b_1(X[i];F)}{[G:G_i]}$ and $\RG(G;(G_i)_{i \ge 0})$ are
      independent of the chain and the characteristic of $F$. It is  possible that the answer is positive also in the
      case, where the characteristic of $F$ is not zero, since in
      Remark~\ref{rem:not_the_L2-Betti_number} the counterexamples of
      Avramidi-Okun-Schreve~\cite[Theorem~1]{Avramidi-Okun-Schreve(2021)} do not work in degree
      $n = 1$.

  \begin{question}\label{que:First_L2_Betti_number_and_cost}
    Let $G$ be a finitely generated group. Do we have
    \[
      b_1^{(2)}(G) = \costoper(G)-1
    \]
    and is the Fixed Price Conjecture true?   (The Fixed Price Conjecture predicts that the cost
    of every standard action of $G$, i.e., an essentially free $G$-action on a standard
    Borel space with $G$-invariant probability measure, is equal to the cost of $G$.)
  \end{question}

  Higher rank versions of the rank gradient are discussed
  in~\cite[Section~5]{Lueck(2016_l2approx)}.


  \subsection{The $L^2$-torsion and applications to group theory}%
  \label{subsec:The_L2-torsion_numbers_and_applications_to_group_theory}%

  As we have explained above, $L^2$-Betti numbers have been exploited for group theory. We
  think that there is a lot of potential for the $L^2$-torsion of a group $G$ to have
  striking applications to group theory, and we encourage group theorists to work on these types of questions.

  A typical question is the following. If $M$ is a closed hyperbolic manifold of odd
  dimension, then $\rho^{(2)}(\pi_1(M))$ is a up to a constant depending only on the dimension the volume of $M$,
  see~\cite{Hess-Schick(1998), Lueck-Schick(1999)}.  Since $\pi_1(M)$ is a word-hyperbolic
  group, one may ask what $\rho^{(2)}(G)$ measures for a $\det$-$L^2$-acyclic torsionfree
  word-hyperbolic group.

  We want to mention the following invariant of an automorphism $f \colon BG \to BG$ of a
  group $G$ for which there exists a finite model for $BG$ and which satisfies the
  Determinant Conjecture~\ref{con:Determinant_Conjecture}.  Then $G \rtimes_f \IZ$ is
  admissible, and hence we can define the $L^2$-torsion of $f$ by
  \begin{equation}\label{rho_upper_(2)(f)}
    \rho^{(2)}(f) := \rho^{(2)}(G \rtimes_f \IZ).
  \end{equation}

  This invariant has the following properties, see~\cite[Theorem~7.27 on
  page~305]{Lueck(2002)}.

  \begin{theorem}\label{the:basic_properties_of_L2-torsion_of_automorphisms}
    Suppose that all groups appearing below have finite $CW$-models for their classifying
    spaces and satisfy the Determinant Conjecture~\ref{con:Determinant_Conjecture}.
    \begin{enumerate}

    \item\label{the:basic_properties_of_L2-torsion_of_automorphisms:amalgamated_products}
      Suppose that $G$ is the amalgamated product $G_1\ast_{G_0} G_2$ for subgroups
      \mbox{$G_i \subset G$} and the automorphism $f\colon G \to G$ is the amalgamated product
      $f_1 \ast_{f_0} f_2$ for automorphisms $f_i\colon G_i \to G_i$. Then
      \[
        \rho^{(2)}(f) = \rho^{(2)}(f_1) + \rho^{(2)}(f_2) - \rho^{(2)}(f_0);
      \]

    \item\label{the:basic_properties_of_L2-torsion_of_automoprhisms:trace_property} Let
      $f\colon G \to H$ and $g\colon H \to G$ be isomorphisms of groups.  Then
      \[
        \rho^{(2)}(f \circ g) = \rho^{(2)}(g \circ f).
      \]
      In particular $\rho^{(2)}(f)$ is invariant under conjugation with automorphisms;

    \item\label{the:basic_properties_of_L2-torsion_of_automoprhisms:additivity} Suppose that
      the following diagram of groups
      \[
        \xymatrix{1 \ar[r] & G_1 \ar[d]^{f_1} \ar[r]^i & G_2 \ar[d]^{f_2} \ar[r]^p & G_3
          \ar[d]^{\id} \ar[r] & 1
          \\
          1 \ar[r] & G_1 \ar[r]^i & G_2 \ar[r]^p & G_3 \ar[r] & 1 }
      \]
      commutes, has exact rows and its vertical arrows are automorphisms. Then
      \[
        \rho^{(2)}(f_2) = \chi(BG_3) \cdot \rho^{(2)}(f_1);
      \]

    \item\label{the:basic_properties_of_L2-torsion_of_automoprhisms:multiplicativity} Let
      $f\colon G \to G$ be an automorphism of a group. Then for all integers $n \ge 1$
      \[
        \rho^{(2)}(f^n) = n \cdot \rho^{(2)}(f);
      \]

    \item\label{the:basic_properties_of_L2-torsion_of_automoprhisms:G_acyclic} We have
      $\rho^{(2)}(f) = 0$, if $G$ satisfies one of the following conditions:

      \begin{enumerate}

      \item We have $b_n^{(2)}(G) = 0$ for every $n \ge 0$;

      \item $G$ contains an amenable infinite normal subgroup.

      \end{enumerate}

    \end{enumerate}
  \end{theorem}

  Let $S$ be a compact connected orientable $2$-dimensional manifold, possibly with
  boundary. Let $f\colon S \to S$ be an orientation preserving homeomorphism. The mapping
  torus $T_f$ is a compact connected orientable $3$-manifold, whose boundary is empty or a
  disjoint union of $2$-dimensional tori. Then there is a maximal family of embedded
  incompressible tori, which are pairwise not isotopic and not boundary parallel, such that
  it decomposes $T_f$ into pieces that are Seifert or hyperbolic.  Let $M_1$, $M_2$,
  $\ldots$, $M_r$ be the hyperbolic pieces.  They all have finite volume $\vol(M_i)$. The
  following result is taken from~\cite[Theorem~7.28 on page~307]{Lueck(2002)}

  \begin{theorem}\label{the:rho(surface_homeomorphism)}
    If $S$ is $S^2$, $D^2$, or $T^2$, then $\rho^{(2)}(f) = 0$. Otherwise we get
    \[
      \rho^{(2)}(\pi_1(f)\colon \pi_1(S) \to \pi_1(S)) = \frac{-1}{6\pi} \cdot \sum_{i=1}^r
      \vol(M_i).
    \]
  \end{theorem}
  The combinatorial approach for the computation of $\rho^{(2)}(f)$ in terms of characteristic
  sequences of Remark~\ref{rem:combinatorial_computation} is described in detail
  in~\cite[Subsection~7.34]{Lueck(2002)}. The favourite and so far unexplored case is,
  when $G$ is a finitely generated free group.

  \begin{question}\label{que:rho_upper_(2)(f)}
    Let $G$ be finitely generated non-abelian free group.

    \begin{enumerate}
    \item Is $\rho^{(2)}(f) \le 0$ for any automorphism $f$ of $G$;
    \item What is the structure of the countable set $\{\rho^{(2)}(f)\}$,
     where $f$ runs through the automorphisms of $G$?
   \item Given a real number $r < 0$, is the set of conjugacy classes of fully irreducible automorphisms
     $f$ of $G$ with $\rho^{(2)}(f) = r$
      finite?
    \end{enumerate}
  \end{question}

  We will discuss measure equivalence in Section~\ref{subsec:measure_equivalence}
  and twisted $L^2$-torsion in
  Section~\ref{sec:Twisting_with_finite-dimensional_representations}.


  \typeout{-------------------------- Section 10: Measure equivalence  -------------------------------}

  \section{Measure equivalence}\label{subsec:measure_equivalence}

   Gaboriau~\cite{Gaboriau(2002a)}
  introduced \emph{$L^2$-Betti numbers of measured equivalence relations}
  and proved that two measure equivalent countable groups
  have proportional $L^2$-Betti numbers.
  This notion turned out to have many important applications
  in recent years, most notably through the work of Popa~\cite{Popa(2007)}.

  The notion of \emph{measure equivalence} was introduced by
  Gromov~\cite[0.5.E]{Gromov(1993)}.

  \begin{definition}\label{def:measure_equivalence}
    Two countable groups $G$ and $H$ are called \emph{measure equivalent
      with index $c=I(G,H)>0$} if there exists a non-trivial standard
    measure space $(\Omega,\mu)$ on which $G\times H$ acts such that the
    restricted actions of $G=G\times\{1\}$ and $H=\{1\}\times H$ have
    measurable fundamental domains $X\subset\Omega$ and
    $Y\subset\Omega$, with $\mu(X)<\infty$, $\mu(Y)<\infty$, and
    $c=\mu(X)/\mu(Y)$. The space $(\Omega,\mu)$ is called a
    \emph{measure coupling} between $G$ and $H$ (of index $c$).
  \end{definition}

  The following conjecture  is taken from~\cite[Conjecture~1.2]{Lueck-Sauer-Wegner(2010)}.

  \begin{conjecture}\label{con:measure-equivalence_and_L2-torsion}
    Let $G$ and $H$ be two admissible groups, which  are measure
    equivalent with index $I(G,H)>0$. Then
    \[
    \rho^{(2)}(G)= I(G,H) \cdot \rho^{(2)}(H).
    \]
  \end{conjecture}

  Due to Gaboriau~\cite{Gaboriau(2002a)}, the vanishing of the $n$th $L^2$-Betti number $b_n^{(2)}(G)$
  is a invariant of the measure equivalence class of a countable group $G$. If all
  $L^2$-Betti numbers vanish and $G$ is an admissible group, then the vanishing of the
  $L^2$-torsion is a secondary invariant of the measure equivalence class of a countable
  group $G$ provided that Conjecture~\ref{con:measure-equivalence_and_L2-torsion}
  holds.

  Evidence for Conjecture~\ref{con:measure-equivalence_and_L2-torsion} comes
  from~\cite[Conjecture~1.10]{Lueck-Sauer-Wegner(2010)} which says that
  Conjecture~\ref{con:measure-equivalence_and_L2-torsion} is true if we replace measure
  equivalence by the stronger notion of uniform measure equivalence,
  see~\cite[Definition~1.3]{Lueck-Sauer-Wegner(2010)}, and assume  that $G$ satisfies the
  Measure Theoretic Determinant Conjecture, see~\cite[Conjecture~1.7]{Lueck-Sauer-Wegner(2010)}.


  \typeout{--------------------- Section 11: Zero-in-the-spectrum-Conjecture -------------------------}

  \section{Zero-in-the-Spectrum-Conjecture}\label{sec:Zero-in-the-Spectrum-Conjecture}

  The next conjecture appears for the first time in
  Gromov's article~\cite[page~120]{Gromov(1986a)}.

  \begin{conjecture}[Zero-in-the-Spectrum Conjecture]%
  \label{con:Zero-in-the-Spectrum-Conjecture}
  Suppose that $\widetilde{M}$ is the universal covering of the
  aspherical closed Riemannian manifold $M$ (with the Riemannian metric
  coming from $M$). Then for some $p \ge 0$ zero is  in the
  spectrum of the minimal closure
  \[
  (\Delta_p)_{\min}\colon  \dom\left((\Delta_p)_{\min}\right) \subset
  L^2\Omega^p(\widetilde{M}) \to L^2\Omega^p(\widetilde{M})
  \]
  of the Laplacian acting on smooth $p$-forms on $\widetilde{M}$.
  \end{conjecture}

  The Zero-in-the-Spectrum Conjecture~\ref{con:Zero-in-the-Spectrum-Conjecture}
  is known to be true if one of the following conditions is satisfied:
  \begin{enumerate}
  \item $\dim(M) \le 3$;
  \item M is a locally symmetric space;
  \item $M$ possesses  a Riemannian metric whose sectional curvature
   is non-positive;
   \item  $M$ is an aspherical closed K\"ahler manifold whose fundamental group is
   word-hyperbolic in the sense of~\cite{Gromov(1987)};
  \item $\pi_1(M)$ satisfies the strong Novikov Conjecture.
  \end{enumerate}

  If one drops the conditions ``aspherical'' in the Zero-in-the-Spectrum
  Conjecture~\ref{con:Zero-in-the-Spectrum-Conjecture}, then there are counterexamples.

  For the  proofs of the claim above and the relevant reference, such
  as~\cite{Farber-Weinberger(2001), Gromov(1986a), Lott(1996b)}, and for more information
  about the Zero-in-the-Spectrum Conjecture~\ref{con:Zero-in-the-Spectrum-Conjecture}, we refer
  to~\cite[Chapter~12]{Lueck(2002)}.


  \typeout{--------------------- Section 12: Twisting with finite-dimensional representations  ------------}

  \section{Twisting with finite-dimensional representations}\label{sec:Twisting_with_finite-dimensional_representations}
  A prominent open problem  is whether one can twist $L^2$-invariants with (not necessarily unitary)
  finite-dimensional representations. A basic and systematical treatment of this problem can be found
  in~\cite{Lueck(2018)}.

  For $L^2$-Betti numbers there is the conjecture that this just boils down to multiplying
  the untwisted $L^2$-Betti number with the dimension of the representation,
  see~\cite[Conjecture~2]{Boschheidgen-Jaikin-Zapirain(2022)}
  and~\cite[Question~0.1]{Lueck(2018)}. This conjecture is proved for sofic groups
  by~Boschheidgen-Jaikin-Zapirian\cite[Theorem~1.1]{Boschheidgen-Jaikin-Zapirain(2022)} and
  for locally indicable groups by Kielak-Sun~\cite[Theorem~4.5]{Kielak-Sun(2021)}.

  For $L^2$-torsion the effect of the twisting is much more interesting and this leads to
  new invariants.  Especially in dimension~$3$ there has been made a lot of progress, and
  interesting open problems occur. For instance, there are interesting relations between the
  Alexander-torsion function, which is given by twisting the $L^2$-torsion with a family of
  one-dimensional representations associated to an element in $H^1(M;\IZ)$, and the Thurston
  polytope.  A prominent open problem is whether the regulare Fuglede-Kadison determinant is
  continuous, see~\cite[Question~9.11]{Lueck(2018)}.  We refer for more information and the
  relevant references, such as~\cite{Dubois-Friedl-Lueck(2016), Friedl-Lueck(2017universal),
    Friedl-Lueck(2019Euler), Friedl-Lueck(2019Thurston), Funke-Kielak(2018), Kielak(2020),
    Linnell-Lueck(2018), Liu(2017)}, to~\cite[Section~10]{Lueck(2018)} and the survey
  article~\cite{Lueck(2021survey)}.


  \typeout{------------------- Section 13:  Von Neumann group algebras over $\IF_p$  -------------------}

  \section{Group von Neumann algebras over $\IF_p$ }\label{sec:Group_von_Neumann_group_algebras_over_IF_p}

  Some of the problems discussed above are of the shape that we understand the case of a
  field of characteristic zero well using the group von Neumann algebra $\caln(G)$, but do
  not know what happens in prime characteristic $p$. A prominent example are the questions
  about approximation of $L^2$-Betti numbers in prime characteristic, see
  Subsection~\ref{subsec:Approximation_of_L2_-Betti_numbers_in_prime_characteristic}.  It
  seems to be conceivable that the relevant limits exists and are independent of the chains,
  but the value of the limits depend on whether we work in characteristic zero or in prime
  characteristic. So the original hope that one always gets as a limit the $L^2$-invariants
  defined in terms of the von Neumann algebra also in prime characteristic turns out not to
  be fullfilled. So we face the new problem what the limit could be in the prime
  characteristic case.

  This raises the question whether there is an $\IF_p$-analogue of the group von Neumann
  algebra? In order to treat $L^2$-Betti numbers in prime characteristic, one would hope for the existence of a
  $\IF_p$-algebra $\caln(G;p)$ which contains the group ring $\IF_p G$ and comes with a
  dimension function for finitely generated projective $\caln(G;p)$-modules
  satisfying~\cite[Assumption~6.2 on page~238]{Lueck(2002)}. Then~\cite[Definition~6.6 and
  Theorem~6.7 on page 239]{Lueck(2002)} would apply, and one would get a dimension function
  for all $\caln(G;p)$-modules which has many useful features. Finally, one would define the
  $n$th $L^2$-Betti number in characteristic $p$ of a $G$-space $X$ to be
  \begin{equation}
    b_n^{(2)}(Y;\caln(G;p)) := \dim_{\caln(G;p)}\bigl(H_n(\caln(G;p) \otimes_{\IF_p G} C^s_*(Y;\IF_p))\bigr)
    \in \IR^{\ge 0} \amalg \{\infty\},
    \label{b_n_upper_(2)(X,caln(G;p))}
  \end{equation}
  where $C^s_*(Y;\IF_p)$ is the $\IF_pG$-chain complex given by the singular chain complex
  with coefficient in $\IF_p$.  For a group $G$ one would define its $n$th $L^2$-Betti
  number in characteristic $p$
  \begin{equation}
    b_n^{(2)}(G;p)  =    b_n^{(2)}(EG;\caln(G;p)).
    \label{b_n_upper_(2)(G,p)}
  \end{equation}
  The hope is that then the corresponding sequences of normalized Betti numbers with
  coefficients in $\IF_p$ converge for any normal chain to the $L^2$-Betti numbers with
  coefficients in $\IF_p$, see also Remark~\ref{rem:embedding_group_ring_skew_field}.


\typeout{------------------- Section 14: $L^2$-invariants and condensed mathematics -------------------}

\section{$L^2$-invariants and condensed mathematics}\label{sec:L2-invariants_and_condensed_mathematics}

Condensed mathematics, a theory recently developed by Clausen-Scholze~\cite{Scholze-Clausen(2019condensed), Scholze-Clausen(2019analyticgeometry), Scholze-Clausen(2022condensedcomplex)}, is a framework which aims to remedy ill behaviour of the category of topological spaces with respect to algebraic structure, for example the fact that the category of topological abelian groups is not an abelian category. It has proven useful in incorporating geometric and analytic structures that appear in arithmetic, or complex geometry and we are far from understanding its full capabilities.
In the context of $L^2$-invariants, condensed mathematics poses at least two interesting questions.
\begin{enumerate}
	\item The theory of $L^2$-invariants heavily relies on techniques mixing algebra and topology or functional analysis.
  Can condensed mathematics help to extend the current formalism of $L^2$-invariants in a way which sheds light on open problems in the area?
  For example, are there analogues of the group von Neumann algebra of a discrete group over $\IF_p$?
	\item Can $L^2$-invariants be applied to a wider class of geometric problems coming from condensed mathematics?
\end{enumerate}
In this section, we attempt to make a first step towards the second question by defining $L^2$-Betti numbers of condensed sets carrying an action of a discrete group in three different ways that extend the current definition for nicely behaved spaces, e.g., CW-complexes:
\begin{enumerate}
  \item homology through solidification;
  \item condensed cohomology;
  \item condensed singular homology.
\end{enumerate}
We would hope that more sophisticated setups to associate $L^2$-Betti numbers to objects in analytic geometry lead to new interesting invariants.
The reader should be aware that we are freely using the language of $\infty$-categories.

Let us begin by giving a brief summary of the main definitions from condensed mathematics necessary for this purpose.
For more details and proofs, we refer the reader to Peter Scholze's lecture notes~\cite{Scholze-Clausen(2019analyticgeometry), Scholze-Clausen(2019condensed), Scholze-Clausen(2022condensedcomplex)}.
We also ignore set theoretic size issues, which are treated with more care in aforementioned references.

Denote by $\edCH$ the category of extremally disconnected compact Hausdorff spaces (extremally disconnected meaning that the closure of any of its open subsets is again open) and continuous maps.
It can be made into a site with coverings given by finite families of jointly surjective maps.
For $\cat{C}$ a category with finite limits, the category $\Cond{\cat{C}}$ of condensed objects in $\cat{C}$ is defined as the category of sheaves on $\edCH$ with values in $\cat{C}$.
It is not hard to check that this identifies with the
full subcategory of $\Fun(\edCH^{\op}, \cat{C})$ of finite product preserving functors.
Explicitly, a condensed object in $\cat{C}$ is a functor
\begin{equation*}
  X \colon \edCH^{\op} \to \cat{C}
\end{equation*}
which satisfies
\begin{enumerate}
\item\label{def:condensed_cald-object:emptyset}
  $X(\emptyset)$ is the terminal object in $\cat{C}$;

\item\label{def:condensed_cald-object:amalg}
  For any two objects $T_1$ and $T_2$ in $\edCH$ the natural map
 \[X(T_1 \amalg T_2) \to X(T_1) \times X(T_2)
 \]
 is an isomorphism.
\end{enumerate}

The category $\Cond{\Set}$ of condensed sets is an enlargement of the category of compactly generated topological spaces in the following sense:
\begin{proposition}[{\cite[Proposition 1.7]{Scholze-Clausen(2019condensed)}}]\label{prop:embedding_top_into_cond}
  The functor $\Top \to \Cond{\Set}, \linebreak X \mapsto \underline X$ given by the restricted Yoneda embedding admits a left adjoint and is fully faithful when restricted to compactly generated topological spaces.
\end{proposition}

The nicely behaved enlargement of topological abelian groups is now given by the category $\Cond{\Ab}$ of condensed abelian groups.
It is a complete and cocomplete abelian category additionally satisfying Grothendieck's axioms (AB5) and (AB6).
We now summarise some additional structure:
\begin{enumerate}
  \item The forgetful functor $\Cond{\Ab} \to \Cond{\Set}$ admits a left adjoint \linebreak $X \mapsto \IZ[X]$.
  \item The category $\Cond{\Ab}$ admits a closed symmetric monoidal structure with $M \otimes N$ given by the sheafification of the presheaf $T \mapsto M(T) \otimes_{\IZ} N(T)$.
  \item There is an equivalence $D(\Cond{\Ab}) \simeq \Cond{D(\IZ)}$ between the derived $\infty$-category of condensed abelian groups and condensed objects in the derived $\infty$-category of $\IZ$-modules.
  The category $\Cond{D(\IZ)}$ is a stable presentably symmetric monoidal $\infty$-category with symmetric monoidal structure induced from the one on $D(\IZ)$.
  It can also be identified with the left derived tensor product on condensed abelian groups.
\end{enumerate}

We now want to outline approach (1) homology through solidification to define $L^2$-Betti numbers of condensed sets.
First we recall some background on solid abelian groups from~\cite[Lecture 5 and 6]{Scholze-Clausen(2019condensed)}.
\begin{definition}
  For a profinite set $T$ define
  \begin{equation*}
    \IZ[T]^\blacksquare = \underline{\Hom}_{\Cond{\Ab}}(\underline{C(T, \IZ)}, \IZ)
  \end{equation*}
  where $C(T, \IZ)$ denotes the space of continuous maps endowed with the compact open topology.
  It comes with a canonical map $\IZ[T] \to \IZ[T]^\blacksquare$.
  Equivalently, if $T = \lim_i T_i$ is given as the cofiltered limit of finite sets, one can identify \linebreak
  $\IZ[T]^\blacksquare \simeq \lim_i \IZ[T_i]$ where the limit is formed in $\Cond{\Ab}$.

  A condensed abelian group $M$ is then called \emph{solid} if for all profinite sets $T$ the induced map
  \begin{equation*}
    \Hom_{\Cond{\Ab}}(\IZ[T]^\blacksquare, M) \to \Hom_{\Cond{\Ab}}(\IZ[T], M)
  \end{equation*}
  is an equivalence.
  Denote by $\Solid{\Ab} \subset \Cond{\Ab}$ the full subcategory of solid abelian groups.

  Similarly, a derived condensed abelian group $C \in D(\Cond{\Ab}) \simeq \Cond{D({\IZ})}$ is called \emph{solid} if for all profinite sets $T$ the induced map
  \begin{equation*}
    \Hom_{D(\Cond{\Ab})}(\IZ[T]^\blacksquare, C) \to \Hom_{D(\Cond{\Ab})}(\IZ[T], C)
  \end{equation*}
  is an equivalence in $D(\IZ)$.
  Denote by $\Solid{D(\IZ)} \subseteq \Cond{D(\IZ)}$ the full subcategory of solid derived $\IZ$-modules.
\end{definition}
The following theorem is the main result from~\cite[Lecture 5]{Scholze-Clausen(2019condensed)}
\begin{theorem}
  The full inclusion $\Solid{D(\IZ)} \subseteq  \Cond{D(\IZ)}$ admits a left adjoint $(-)^{L\blacksquare}$ called the solidification.
  Furthermore, a condensed derived $\IZ$-module is solid if and only if all of its homology groups are solid abelian groups.
\end{theorem}

Now we come to the definition of $L^2$-Betti numbers of condensed sets.
Let $X$ be a condensed set with $G$-action, i.e., an element in $\Cond{\Set^{BG}} \simeq \Cond{\Set}^{BG}$.
Consider the following composition
\begin{align}\label{eq:condensed_chain_complex}
  & \Cond{\Set}^{BG} \xrightarrow{\IZ[-]} \Cond{D(\IZ)}^{BG} \xrightarrow{(-)^{L\blacksquare}}  \Cond{D(\IZ)}^{BG} \\
  \simeq \ & \Cond{D(\IZ)^{BG}} \simeq \Cond{D(\IZ[G])} \nonumber
\end{align}
and denote the image of $X$ under~\eqref{eq:condensed_chain_complex} by $\IZ[X]^{L\blacksquare}$.
\begin{definition}\label{def:equivariant_condensed_homology}
Let $R$ be a $\IZ[G]$-algebra and $X$ a condensed set with $G$-action.
We define solid equivariant homology of $X$ with coefficients in $R$ by
\begin{equation}
  H_n^{G, \blacksquare}(X; R) \coloneqq \Gamma(H_n(\IZ[X]^{L\blacksquare} \otimes_{\IZ[G]} R)).
\end{equation}
Here, $\Gamma \colon \Cond{\Mod_{R}} \to \Mod_{R}$ is the global sections functor and
$- \otimes_{\IZ[G]} R \colon D(\IZ[G]) \linebreak \to D(R)$ denotes the (derived) base change.

Furthermore, we define the $L^2$-Betti numbers of $X$ by
\begin{equation}\label{def:L2_betti_numbers_condensed_anima}
  b^{(2)}_n(X; \caln(G)) \coloneqq \dim_{\caln(G)}(H_n^{G, \blacksquare}(X; \caln(G))).
\end{equation}
\end{definition}

For free $G$-CW-complexes, we now compare equivariant homology and $L^2$-Betti numbers from Definition~\ref{def:equivariant_condensed_homology} with the classical ones, keeping Proposition~\ref{prop:embedding_top_into_cond} in mind.
For this we need one further ingredient, giving a description of the singular chain complex of a CW-complex in the condensed world.
\begin{proposition}[{\cite[Example 6.5]{Scholze-Clausen(2019condensed)}}]\label{prop:condensed_singular_homology}
  There is an equivalence of the two functors
  \begin{align*}
    & \CW \xrightarrow{\underline{(-)}} \Cond{\Set} \xrightarrow{\IZ[-]} \Cond{D(\IZ)} \xrightarrow{(-)^{L\blacksquare}} \Cond{D(\IZ)} \ \text{and} \\
    & 
    \CW \xrightarrow{C_\bullet(-)} D(\IZ) \xrightarrow{\underline{(-)}} \Cond{D(\IZ)},
  \end{align*}
  where $\CW \subseteq \Top$ denotes the full subcategory of CW-complexes, $C_\bullet(X)$ denotes the singular chain complex of $X$, and the functor $\underline{(-)} \colon D(\IZ) \to \Cond{D(\IZ)}$ sends an object $M$ to the sheafification of the constant presheaf with value $M$.
\end{proposition}

We now explain why this generalises the $L^2$-Betti number of $G$-CW-complexes defined in the body of this article.
\begin{theorem}\label{thm:comparison_condensed_classical_L2}
  Let $X$ be a $G$-space which is $G$-homotopy equivalent to a $G$-CW-complex and $R$ a $\IZ[G]$-algebra.
  Then Borel homology of $X$ with coefficients in $R$ and homology of $X$ in the sense of Definition~\ref{def:equivariant_condensed_homology} are equivalent.
  In particular, if $X$ is $G$-homotopy equivalent to a free $G$-CW-complex, the $L^2$-Betti numbers of $X$ in the sense of~\eqref{b_n_upper_(2)(X,caln(G))} and Definition~\ref{def:L2_betti_numbers_condensed_anima} agree.
\end{theorem}
\begin{proof}
  Consider the diagram
  \begin{equation}
  \begin{tikzcd}
    \CW^{BG} \ar[r, "{C_\bullet(-)}"] \ar[d, "{\IZ[-]}"] & D(\IZ)^{BG} \ar[r, "\simeq"] \ar[d, "\underline{(-)}"] & D(\IZ[G]) \ar[r, "{- \otimes_{\IZ[G]} R}"] \ar[d, "\underline{(-)}"] & D(R) \ar[d, "\underline{(-)}"] \\
    \Cond{D(\IZ)}^{BG} \ar[r, "{(-)^{L\blacksquare}}"] & \Cond{D(\IZ)}^{BG} \ar[r, "\simeq"] & \Cond{D(\IZ[G])} \ar[r, "{- \otimes_{\IZ[G]} R}"] & \Cond{D(R)}.
  \end{tikzcd}
  \end{equation}
  The left square commutes by Proposition~\ref{prop:condensed_singular_homology} and the rest of the diagram commutes for obvious reasons.
  The composite of the upper horizontal arrows recovers Borel homology of a $G$-CW-complex.
  It remains to note that taking global sections $\Gamma \colon \Cond{D(R)} \to D(R)$ is left inverse to the functor $\underline{(-)} \colon D(R) \to \Cond{D(R)}$.

  For a free $G$-CW-complex $X$, Borel homology of $X$ with coefficients in $R$ can be computed as homology of the degreewise tensor product of the singular chain complex of $X$ with $R$, as used in the definition of $L^2$-Betti numbers in~\cite[Definition 6.50 on page 263]{Lueck(2002)}.

  To extend this result to all $G$-spaces $G$-homotopy equivalent to free $G$-CW-complexes, observe that the functor $\Top \to \Cond{D(\IZ)}, X \mapsto \IZ[\underline X]^{L\blacksquare}$ sends homotopic maps to homotopic maps which formally follows from $\IZ[[0,1]]^{L\blacksquare} \simeq \IZ$.
\end{proof}

\begin{remark}[General actions]
  Note that the above definition does not recover $L^2$-Betti numbers for non-free $G$-CW-complexes.
  To remedy this, one can work parametrised over the orbit category $\Or(G)$ instead of $BG$.
\end{remark}

We now briefly explain the two alternative definitions of $L^2$-Betti numbers of condensed sets with $G$-action mentioned in the beginning of Section~\ref{sec:L2-invariants_and_condensed_mathematics}.

\begin{remark}[Condensed singular homology]
   One alternative is to model the construction of singular homology inside condensed sets defined as homology of the simplicial set
   \begin{equation*}
     \Hom_{\Cond{\Set}}(\underline{\Delta^\bullet}, X).
   \end{equation*}
   This in fact agrees with singular homology for all topological spaces since the spaces $\Delta^n$ are compactly generated so that the realization of the condensed set $\underline{\Delta^n}$ is $\Delta^n$.
\end{remark}

\begin{remark}[A cohomological variant]
  One can also define equivariant cohomology of condensed $G$-sets analogous to Definition~\ref{def:equivariant_condensed_homology}.
  For this, one assigns to a condensed set $X$ with $G$-action the homology of
  \begin{equation*}
    \Hom_{\Cond{D(\IZ[G])}}(\IZ[X]^{L\blacksquare}, \underline R) \in D(R).
  \end{equation*}
  Then the same arguments as in Theorem~\ref{thm:comparison_condensed_classical_L2} show that for free $G$-CW-complexes this recovers Borel cohomology.

  Note that as for any abelian group the induced constant condensed abelian group is solid we have
  \begin{align*}
    \Hom_{\Cond{D(\IZ[G])}}(\IZ[X]^{L\blacksquare}, \underline R) & \simeq \Hom_{\Cond{D(\IZ)}}(\IZ[X]^{L\blacksquare}, \underline R)^{hG} \\
    & \simeq \Hom_{\Cond{D(\IZ)}}(\IZ[X], \underline R)^{hG} \\
    & \simeq \Hom_{\Cond{D(\IZ[G])}}(\IZ[X], \underline R). \label{eq:condensed_cohomology_no_solidification}
  \end{align*}
  Because of that, this version of equivariant condensed cohomology looks most natural to us.
  From the viewpoint of $L^2$-invariants, working with homology instead of cohomology is often easier (for instance homology behaves better with respect to colimits) and chose to present the dual version in more detail.
\end{remark}



\begin{thebibliography}{100}

\bibitem{Abert-Bergeron-Biringer-Gelander-Nikolov-Raimbault-Samet(2017)}
M.~Abert, N.~Bergeron, I.~Biringer, T.~Gelander, N.~Nikolov, J.~Raimbault, and
  I.~Samet.
\newblock On the growth of {$L^2$}-invariants for sequences of lattices in
  {L}ie groups.
\newblock {\em Ann. of Math. (2)}, 185, 2017.

\bibitem{Abert-Bergeron-Fraczyk-Gaboriau(2021)}
M.~Abert, N.~Bergeron, M.~Fraczyk, and D.~Gaboriau.
\newblock On homology torsion growth.
\newblock Preprint, arXiv:2106.13051 [math.GT], 2021.

\bibitem{Abert-Jaikin-Zapirain-Nikolov(2011)}
M.~Ab{\'e}rt, A.~Jaikin-Zapirain, and N.~Nikolov.
\newblock The rank gradient from a combinatorial viewpoint.
\newblock {\em Groups Geom. Dyn.}, 5(2):213--230, 2011.

\bibitem{Abert-Nikolov(2012)}
M.~Ab{\'e}rt and N.~Nikolov.
\newblock Rank gradient, cost of groups and the rank versus {H}eegaard genus
  problem.
\newblock {\em J. Eur. Math. Soc. (JEMS)}, 14(5):1657--1677, 2012.

\bibitem{Agol(2008)}
I.~Agol.
\newblock Criteria for virtual fibering.
\newblock {\em J. Topol.}, 1(2):269--284, 2008.

\bibitem{Albanese-Di-Cerbo-Lombardi(2023)}
M.~Albanese, L.~Di~Cerbo, and L.~Lombardi.
\newblock Aspherical complex surfaces, the {S}inger {C}onjecture, and
  {G}romov's inequality {$\chi \geq |\sigma|$}.
\newblock Preprint, arXiv:2311.10226 [math.DG], 2023.

\bibitem{Andrew-Guerch-Hughes-Kudlinska(2023)}
N.~Andrew, Y.~Guerch, S.~Hughes, and M.~Kudlinska.
\newblock Homology growth of polynomially growing mapping tori.
\newblock Preprint, arXiv:2305.10410 [math.GR], 2023.

\bibitem{Andrew-Hughes-Kudlinska(2022)}
N.~Andrew, S.~Hughes, and M.~Kudlinska.
\newblock Torsion homology growth of polynomially growing free-by-cyclic
  groups.
\newblock Preprint, arXiv:2211.04389 [math.GR], 2022.

\bibitem{Atiyah(1976)}
M.~F. Atiyah.
\newblock Elliptic operators, discrete groups and von {N}eumann algebras.
\newblock {\em Ast\'erisque}, 32-33:43--72, 1976.

\bibitem{Austin(2013)}
T.~Austin.
\newblock Rational group ring elements with kernels having irrational
  dimension.
\newblock {\em Proc. Lond. Math. Soc. (3)}, 107(6):1424--1448, 2013.

\bibitem{Avramidi(2018)}
G.~Avramidi.
\newblock Rational manifold models for duality groups.
\newblock {\em Geom. Funct. Anal.}, 28(4):965--994, 2018.

\bibitem{Avramidi-Davis-Okun-Schreve(2016)}
G.~Avramidi, M.~W. Davis, B.~Okun, and K.~Schreve.
\newblock The action dimension of right-angled {A}rtin groups.
\newblock {\em Bull. Lond. Math. Soc.}, 48(1):115--126, 2016.

\bibitem{Avramidi-Okun-Schreve(2021)}
G.~Avramidi, B.~Okun, and K.~Schreve.
\newblock Mod {$p$} and torsion homology growth in nonpositive curvature.
\newblock {\em Invent. Math.}, 226(3):711--723, 2021.

\bibitem{Ballmann-Bruening(2001)}
W.~Ballmann and J.~Br{\"u}ning.
\newblock On the spectral theory of manifolds with cusps.
\newblock {\em J. Math. Pures Appl. (9)}, 80(6):593--625, 2001.

\bibitem{Bergeron-Linnell-Lueck-Sauer(2014)}
N.~{Bergeron}, P.~{Linnell}, W.~{L\"uck}, and R.~{Sauer}.
\newblock {On the growth of Betti numbers in $p$-adic analytic towers.}
\newblock {\em {Groups Geom. Dyn.}}, 8(2):311--329, 2014.

\bibitem{Bergeron-Venkatesh(2013)}
N.~Bergeron and A.~Venkatesh.
\newblock The asymptotic growth of torsion homology for arithmetic groups.
\newblock {\em J. Inst. Math. Jussieu}, 12(2):391--447, 2013.

\bibitem{Bestvina-Kapovich-Kleiner(2002)}
M.~Bestvina, M.~Kapovich, and B.~Kleiner.
\newblock Van {K}ampen's embedding obstruction for discrete groups.
\newblock {\em Invent. Math.}, 150(2):219--235, 2002.

\bibitem{Boschheidgen-Jaikin-Zapirain(2022)}
J.~Boschheidgen and A.~Jaikin-Zapirain.
\newblock Twisted {$L^2$}-{B}etti numbers of sofic groups.
\newblock Preprint, arXiv:2201.03268 [math.GR], 2022.

\bibitem{Bridson-Kochloukova(2017)}
M.~R. Bridson and D.~H. Kochloukova.
\newblock Volume gradients and homology in towers of residually-free groups.
\newblock {\em Math. Ann.}, 367(3-4):1007--1045, 2017.

\bibitem{Burghelea-Friedlander-Kappeler-McDonald(1996a)}
D.~Burghelea, L.~Friedlander, T.~Kappeler, and P.~McDonald.
\newblock Analytic and {R}eidemeister torsion for representations in finite
  type {H}ilbert modules.
\newblock {\em Geom. Funct. Anal.}, 6(5):751--859, 1996.

\bibitem{Calegari-Emerton(2009bounds)}
F.~Calegari and M.~Emerton.
\newblock Bounds for multiplicities of unitary representations of cohomological
  type in spaces of cusp forms.
\newblock {\em Ann. of Math. (2)}, 170(3):1437--1446, 2009.

\bibitem{Calegari-Emerton(2011)}
F.~Calegari and M.~Emerton.
\newblock Mod-{$p$} cohomology growth in {$p$}-adic analytic towers of
  3-manifolds.
\newblock {\em Groups Geom. Dyn.}, 5(2):355--366, 2011.

\bibitem{Cheeger-Gromov(1986)}
J.~Cheeger and M.~Gromov.
\newblock ${L}\sb 2$-cohomology and group cohomology.
\newblock {\em Topology}, 25(2):189--215, 1986.

\bibitem{Clay(2017free)}
M.~Clay.
\newblock {$\ell^2$}-torsion of free-by-cyclic groups.
\newblock {\em Q. J. Math.}, 68(2):617--634, 2017.

\bibitem{Davis-Okun(2001)}
M.~W. Davis and B.~Okun.
\newblock Vanishing theorems and conjectures for the {$\ell\sp 2$}-homology of
  right-angled {C}oxeter groups.
\newblock {\em Geom. Topol.}, 5:7--74 (electronic), 2001.

\bibitem{Dodziuk-Linnell-Mathai-Schick_Yates(2003)}
J.~Dodziuk, P.~Linnell, V.~Mathai, T.~Schick, and S.~Yates.
\newblock Approximating {$L^2$}-invariants and the {A}tiyah conjecture.
\newblock {\em Comm. Pure Appl. Math.}, 56(7):839--873, 2003.
\newblock Dedicated to the memory of J{\"u}rgen K. Moser.

\bibitem{Dubois-Friedl-Lueck(2016)}
J.~Dubois, S.~Friedl, and W.~L{\"u}ck.
\newblock The {$L^2$}-{A}lexander torsion of 3-manifolds.
\newblock {\em J. Topol.}, 9(3):889--926, 2016.

\bibitem{Elek-Szabo(2005)}
G.~Elek and E.~Szab{\'o}.
\newblock Hyperlinearity, essentially free actions and {$L^2$}-invariants.
  {T}he sofic property.
\newblock {\em Math. Ann.}, 332(2):421--441, 2005.

\bibitem{Elek-Szabo(2006)}
G.~Elek and E.~Szab{\'o}.
\newblock On sofic groups.
\newblock {\em J. Group Theory}, 9(2):161--171, 2006.

\bibitem{Ershov-Lueck(2014)}
M.~Ershov and W.~L{\"u}ck.
\newblock The first $l^2$-{B}etti number and approximation in arbitrary
  characteristics.
\newblock {\em Documenta}, 19:313--331, 2014.

\bibitem{Farber-Weinberger(2001)}
M.~Farber and S.~Weinberger.
\newblock On the zero-in-the-spectrum conjecture.
\newblock {\em Ann. of Math. (2)}, 154(1):139--154, 2001.

\bibitem{Farkas-Linnell(2006)}
D.~R. Farkas and P.~A. Linnell.
\newblock Congruence subgroups and the {A}tiyah conjecture.
\newblock In {\em Groups, rings and algebras}, volume 420 of {\em Contemp.
  Math.}, pages 89--102. Amer. Math. Soc., Providence, RI, 2006.

\bibitem{Friedl-Lueck(2017universal)}
S.~Friedl and W.~L\"uck.
\newblock Universal {$L^2$}-torsion, polytopes and applications to 3-manifolds.
\newblock {\em Proc. Lond. Math. Soc. (3)}, 114(6):1114--1151, 2017.

\bibitem{Friedl-Lueck(2019Euler)}
S.~Friedl and W.~L\"{u}ck.
\newblock {$L^2$}-{E}uler characteristics and the {T}hurston norm.
\newblock {\em Proc. Lond. Math. Soc. (3)}, 118(4):857--900, 2019.

\bibitem{Friedl-Lueck(2019Thurston)}
S.~Friedl and W.~L\"{u}ck.
\newblock The {$L^2$}-torsion function and the {T}hurston norm of 3-manifolds.
\newblock {\em Comment. Math. Helv.}, 94(1):21--52, 2019.

\bibitem{Frigerio-Loeh-Pagliantini_Sauer(2016)}
R.~Frigerio, C.~L\"{o}h, C.~Pagliantini, and R.~Sauer.
\newblock Integral foliated simplicial volume of aspherical manifolds.
\newblock {\em Israel J. Math.}, 216(2):707--751, 2016.

\bibitem{Funke-Kielak(2018)}
F.~Funke and D.~Kielak.
\newblock Alexander and {T}hurston norms, and the {B}ieri-{N}eumann-{S}trebel
  invariants for free-by-cyclic groups.
\newblock {\em Geom. Topol.}, 22(5):2647--2696, 2018.

\bibitem{Gaboriau(2000b)}
D.~Gaboriau.
\newblock Co\^ut des relations d'\'equivalence et des groupes.
\newblock {\em Invent. Math.}, 139(1):41--98, 2000.

\bibitem{Gaboriau(2002a)}
D.~Gaboriau.
\newblock Invariants {$l\sp 2$} de relations d'\'equivalence et de groupes.
\newblock {\em Publ. Math. Inst. Hautes \'Etudes Sci.}, 95:93--150, 2002.

\bibitem{Gaboriau(2002b)}
D.~Gaboriau.
\newblock On orbit equivalence of measure preserving actions.
\newblock In {\em Rigidity in dynamics and geometry (Cambridge, 2000)}, pages
  167--186. Springer, Berlin, 2002.

\bibitem{Grabowski(2014Turing)}
{\L}.~Grabowski.
\newblock On {T}uring dynamical systems and the {A}tiyah problem.
\newblock {\em Invent. Math.}, 198(1):27--69, 2014.

\bibitem{Gromov(1982)}
M.~Gromov.
\newblock Volume and bounded cohomology.
\newblock {\em Inst. Hautes \'Etudes Sci. Publ. Math.}, 56:5--99 (1983), 1982.

\bibitem{Gromov(1986a)}
M.~Gromov.
\newblock Large {R}iemannian manifolds.
\newblock In {\em Curvature and topology of Riemannian manifolds (Katata,
  1985)}, pages 108--121. Springer-Verlag, Berlin, 1986.

\bibitem{Gromov(1987)}
M.~Gromov.
\newblock Hyperbolic groups.
\newblock In {\em Essays in group theory}, pages 75--263. Springer-Verlag, New
  York, 1987.

\bibitem{Gromov(1991)}
M.~Gromov.
\newblock K\"ahler hyperbolicity and ${L}\sb 2$-{H}odge theory.
\newblock {\em J. Differential Geom.}, 33(1):263--292, 1991.

\bibitem{Gromov(1993)}
M.~Gromov.
\newblock Asymptotic invariants of infinite groups.
\newblock In {\em Geometric group theory, Vol.\ 2 (Sussex, 1991)}, pages
  1--295. Cambridge Univ. Press, Cambridge, 1993.

\bibitem{Gromov(1999a)}
M.~Gromov.
\newblock {\em Metric structures for {R}iemannian and non-{R}iemannian spaces}.
\newblock Birkh\"auser Boston Inc., Boston, MA, 1999.
\newblock Based on the 1981 French original [MR 85e:53051], With appendices by
  M.\ Katz, P.\ Pansu and S.\ Semmes, Translated from the French by
  S.~M.~Bates.

\bibitem{Henneke(2021PhD)}
F.~Henneke.
\newblock Embeddings of group rings and {$L^2$}-invariants.
\newblock PhD-thesis, Bonn, 2021.

\bibitem{Henneke-Kielak(2021)}
F.~Henneke and D.~Kielak.
\newblock Agrarian and {$L^2$}-invariants.
\newblock {\em Fund. Math.}, 255(3):255--287, 2021.

\bibitem{Hess-Schick(1998)}
E.~Hess and T.~Schick.
\newblock ${L}\sp 2$-torsion of hyperbolic manifolds.
\newblock {\em Manuscripta Math.}, 97(3):329--334, 1998.

\bibitem{Hirzebruch(1970)}
F.~Hirzebruch.
\newblock The signature theorem: reminiscences and recreation.
\newblock In {\em Prospects in mathematics (Proc. Sympos., Princeton Univ.,
  Princeton, N.J., 1970)}, pages 3--31. Ann. of Math. Studies, No. 70.
  Princeton Univ. Press, Princeton, N.J., 1971.

\bibitem{Ivanov(1987)}
N.~Ivanov.
\newblock Foundations of the theory of bounded cohomology.
\newblock {\em J. Soviet Math.}, 37:1090--1114, 1987.

\bibitem{Jaikin-Zapirain(2019)}
A.~Jaikin-Zapirain.
\newblock The base change in the {A}tiyah and the {L}\"{u}ck approximation
  conjectures.
\newblock {\em Geom. Funct. Anal.}, 29(2):464--538, 2019.

\bibitem{Jaikin-Zapirain(2017positive)}
A.~Jaikin-Zapirain.
\newblock {$L^2$}-{B}etti numbers and their analogues in positive
  characteristic.
\newblock In {\em Groups {S}t {A}ndrews 2017 in {B}irmingham}, volume 455 of
  {\em London Math. Soc. Lecture Note Ser.}, pages 346--405. Cambridge Univ.
  Press, Cambridge, 2019.

\bibitem{Jaikin-Zapirain(2021)}
A.~Jaikin-Zapirain.
\newblock The universality of {H}ughes-free division rings.
\newblock {\em Selecta Math. (N.S.)}, 27(4):Paper No. 74, 33, 2021.

\bibitem{Jaikin-Zapirain+Lopez-Alvarez(2020)}
A.~Jaikin-Zapirain and D.~L\'{o}pez-\'{A}lvarez.
\newblock The strong {A}tiyah and {L}\"{u}ck approximation conjectures for
  one-relator groups.
\newblock {\em Math. Ann.}, 376(3-4):1741--1793, 2020.

\bibitem{Jost-Xin(2000)}
J.~Jost and Y.~L. Xin.
\newblock Vanishing theorems for ${L}\sp 2$-cohomology groups.
\newblock {\em J. Reine Angew. Math.}, 525:95--112, 2000.

\bibitem{Kammeyer(2019)}
H.~Kammeyer.
\newblock {\em Introduction to {$\ell^2$}-invariants}, volume 2247 of {\em
  Lecture Notes in Mathematics}.
\newblock Springer, Cham, 2019.

\bibitem{Kammeyer-Kionke(2023)}
H.~Kammeyer and S.~Kionke.
\newblock On the profinite rigidity of lattices in higher rank {L}ie groups.
\newblock {\em Math. Proc. Cambridge Philos. Soc.}, 174(2):369--384, 2023.

\bibitem{Kielak(2020)}
D.~Kielak.
\newblock The {B}ieri-{N}eumann-{S}trebel invariants via {N}ewton polytopes.
\newblock {\em Invent. Math.}, 219(3):1009--1068, 2020.

\bibitem{Kielak(2020fibring)}
D.~Kielak.
\newblock Residually finite rationally solvable groups and virtual fibring.
\newblock {\em J. Amer. Math. Soc.}, 33(2):451--486, 2020.

\bibitem{Kielak-Linton(2023Atiyah)}
D.~Kielak and M.~Linton.
\newblock The {A}tiyah conjecture for three-manifold groups.
\newblock Preprint, arXiv:2303.15907 [math.GT], 2023.

\bibitem{Kielak-Sun(2021)}
D.~Kielak and B.~Sun.
\newblock Agrarian and {$L^2$}-{B}etti numbers of locally indicable groups,
  with a twist.
\newblock Preprint,arXiv:2112.07394 [math.GT], 2021.

\bibitem{Kleiner-Lott(2008)}
B.~Kleiner and J.~Lott.
\newblock Notes on {P}erelman's papers.
\newblock {\em Geom. Topol.}, 12(5):2587--2855, 2008.

\bibitem{Knebusch-Linnell-Schick(2017)}
A.~Knebusch, P.~Linnell, and T.~Schick.
\newblock On the center-valued {A}tiyah conjecture for {$L^2$}-{B}etti numbers.
\newblock {\em Doc. Math.}, 22:659--677, 2017.

\bibitem{Lackenby(2005expanders)}
M.~Lackenby.
\newblock Expanders, rank and graphs of groups.
\newblock {\em Israel J. Math.}, 146:357--370, 2005.

\bibitem{Lackenby(2009propery_tau)}
M.~Lackenby.
\newblock Large groups, property {$(\tau)$} and the homology growth of
  subgroups.
\newblock {\em Math. Proc. Cambridge Philos. Soc.}, 146(3):625--648, 2009.

\bibitem{Lackenby(2010large)}
M.~Lackenby.
\newblock Detecting large groups.
\newblock {\em J. Algebra}, 324(10):2636--2657, 2010.

\bibitem{Le(2014Mahler)}
T.~Le.
\newblock Homology torsion growth and {M}ahler measure.
\newblock {\em Comment. Math. Helv.}, 89(3):719--757, 2014.

\bibitem{Li-Thom(2014)}
H.~Li and A.~Thom.
\newblock Entropy, determinants, and {$L^2$}-torsion.
\newblock {\em J. Amer. Math. Soc.}, 27(1):239--292, 2014.

\bibitem{Linnell-Lueck(2018)}
P.~Linnell and W.~L\"uck.
\newblock Localization, {W}hitehead groups and the {A}tiyah conjecture.
\newblock {\em Ann. K-Theory}, 3(1):33--53, 2018.

\bibitem{Linnell-Lueck-Sauer(2011)}
P.~Linnell, W.~L{\"u}ck, and R.~Sauer.
\newblock The limit of {$\Bbb F_p$}-{B}etti numbers of a tower of finite covers
  with amenable fundamental groups.
\newblock {\em Proc. Amer. Math. Soc.}, 139(2):421--434, 2011.

\bibitem{Linnell-Okun-Schick(2012)}
P.~Linnell, B.~Okun, and T.~Schick.
\newblock The strong {A}tiyah conjecture for right-angled {A}rtin and {C}oxeter
  groups.
\newblock {\em Geom. Dedicata}, 158:261--266, 2012.

\bibitem{Linnell-Schick(2007)}
P.~Linnell and T.~Schick.
\newblock Finite group extensions and the {A}tiyah conjecture.
\newblock {\em J. Amer. Math. Soc.}, 20(4):1003--1051 (electronic), 2007.

\bibitem{Linnell(1993)}
P.~A. Linnell.
\newblock Division rings and group von {N}eumann algebras.
\newblock {\em Forum Math.}, 5(6):561--576, 1993.

\bibitem{Liu(2017)}
Y.~Liu.
\newblock Degree of {$L^2$}-{A}lexander torsion for 3-manifolds.
\newblock {\em Invent. Math.}, 207(3):981--1030, 2017.

\bibitem{Loeh-Uschold(2022)}
C.~L{\"o}h and M.~Uschold.
\newblock {$L^2$}-{B}etti numbers and computability of reals.
\newblock Preprint, arXiv:2202.03159 [math.GR], 2022.

\bibitem{Lott(1992a)}
J.~Lott.
\newblock Heat kernels on covering spaces and topological invariants.
\newblock {\em J. Differential Geom.}, 35(2):471--510, 1992.

\bibitem{Lott(1996b)}
J.~Lott.
\newblock The zero-in-the-spectrum question.
\newblock {\em Enseign. Math. (2)}, 42(3-4):341--376, 1996.

\bibitem{Lueck(1994c)}
W.~L{\"u}ck.
\newblock Approximating ${L}\sp 2$-invariants by their finite-dimensional
  analogues.
\newblock {\em Geom. Funct. Anal.}, 4(4):455--481, 1994.

\bibitem{Lueck(1994b)}
W.~L{\"u}ck.
\newblock ${L}\sp 2$-{B}etti numbers of mapping tori and groups.
\newblock {\em Topology}, 33(2):203--214, 1994.

\bibitem{Lueck(1994a)}
W.~L{\"u}ck.
\newblock ${L}\sp 2$-torsion and $3$-manifolds.
\newblock In {\em Low-dimensional topology (Knoxville, TN, 1992)}, pages
  75--107. Internat. Press, Cambridge, MA, 1994.

\bibitem{Lueck(1998a)}
W.~L{\"u}ck.
\newblock Dimension theory of arbitrary modules over finite von {N}eumann
  algebras and ${L}\sp 2$-{B}etti numbers. {I}. {F}oundations.
\newblock {\em J. Reine Angew. Math.}, 495:135--162, 1998.

\bibitem{Lueck(1998b)}
W.~L{\"u}ck.
\newblock Dimension theory of arbitrary modules over finite von {N}eumann
  algebras and ${L}\sp 2$-{B}etti numbers. {I}{I}. {A}pplications to
  {G}rothendieck groups, ${L}\sp 2$-{E}uler characteristics and {B}urnside
  groups.
\newblock {\em J. Reine Angew. Math.}, 496:213--236, 1998.

\bibitem{Lueck(2002)}
W.~L{\"u}ck.
\newblock {\em {$L\sp 2$}-{I}nvariants: {T}heory and {A}pplications to
  {G}eometry and \mbox{{$K$}-{T}heory}}, volume~44 of {\em Ergebnisse der
  Mathematik und ihrer Grenzgebiete. 3.~Folge. A Series of Modern Surveys in
  Mathematics [Results in Mathematics and Related Areas. 3rd Series. A Series
  of Modern Surveys in Mathematics]}.
\newblock Springer-Verlag, Berlin, 2002.

\bibitem{Lueck(2005s)}
W.~L{\"u}ck.
\newblock Survey on classifying spaces for families of subgroups.
\newblock In {\em Infinite groups: geometric, combinatorial and dynamical
  aspects}, volume 248 of {\em Progr. Math.}, pages 269--322. Birkh\"auser,
  Basel, 2005.

\bibitem{Lueck(2013l2approxfib)}
W.~L{\"u}ck.
\newblock Approximating {$L^2$}-invariants and homology growth.
\newblock {\em Geom. Funct. Anal.}, 23(2):622--663, 2013.

\bibitem{Lueck(2016_l2approx)}
W.~L{\"u}ck.
\newblock Approximating {$L^2$}-invariants by their classical counterparts.
\newblock {\em EMS Surv. Math. Sci.}, 3(2):269--344, 2016.

\bibitem{Lueck(2018)}
W.~L\"{u}ck.
\newblock Twisting $l^2$-invariants with finite-dimensional representations.
\newblock {\em J. Topol. Anal.}, 10(4):723--816, 2018.

\bibitem{Lueck(2021survey)}
W.~L\"{u}ck.
\newblock Survey on {$L^2$}-invariants and 3-manifolds.
\newblock {\em Bull. Lond. Math. Soc.}, 53(6):1583--1620, 2021.

\bibitem{Lueck(2022book)}
W.~L\"uck.
\newblock {I}somorphism {C}onjectures in {$K$}- and {$L$}-theory.
\newblock in preparation, see http://www.him.uni-bonn.de/lueck/data/ic.pdf,
  2024.

\bibitem{Lueck-Osin(2011)}
W.~L{\"u}ck and D.~Osin.
\newblock Approximating the first {$L^2$}-{B}etti number of residually finite
  groups.
\newblock {\em J. Topol. Anal.}, 3(2):153--160, 2011.

\bibitem{Lueck-Rothenberg(1991)}
W.~L{\"u}ck and M.~Rothenberg.
\newblock Reidemeister torsion and the ${K}$-theory of von {N}eumann algebras.
\newblock {\em $K$-Theory}, 5(3):213--264, 1991.

\bibitem{Lueck-Sauer-Wegner(2010)}
W.~L\"uck, R.~Sauer, and C.~Wegner.
\newblock ${L^2}$-torsion, the measure-theoretic determinant conjecture, and
  uniform measure equivalence.
\newblock {\em Journal of Topology and Analysis}, 2 (2):145--171, 2010.

\bibitem{Lueck-Schick(1999)}
W.~L{\"u}ck and T.~Schick.
\newblock ${L^2}$-torsion of hyperbolic manifolds of finite volume.
\newblock {\em GAFA}, 9:518--567, 1999.

\bibitem{Lueck-Schick(2003)}
W.~L{\"u}ck and T.~Schick.
\newblock Various {$L\sp 2$}-signatures and a topological \mbox{{$L\sp
  2$}-signature} theorem.
\newblock In {\em High-dimensional manifold topology}, pages 362--399. World
  Sci. Publishing, River Edge, NJ, 2003.

\bibitem{Lueck-Schick(2005)}
W.~L{\"u}ck and T.~Schick.
\newblock Approximating {$L\sp 2$}-signatures by their compact analogues.
\newblock {\em Forum Math.}, 17(1):31--65, 2005.

\bibitem{Mathai(1992)}
V.~Mathai.
\newblock ${L}\sp 2$-analytic torsion.
\newblock {\em J. Funct. Anal.}, 107(2):369--386, 1992.

\bibitem{Morgan-Tian(2014)}
J.~Morgan and G.~Tian.
\newblock {\em The geometrization conjecture}, volume~5 of {\em Clay
  Mathematics Monographs}.
\newblock American Mathematical Society, Providence, RI; Clay Mathematics
  Institute, Cambridge, MA, 2014.

\bibitem{Okun-Schreve(2016)}
B.~Okun and K.~Schreve.
\newblock The {$L^2$}-(co)homology of groups with hierarchies.
\newblock {\em Algebr. Geom. Topol.}, 16(5):2549--2569, 2016.

\bibitem{Olbrich(2002)}
M.~Olbrich.
\newblock {$L^2$}-invariants of locally symmetric spaces.
\newblock {\em Doc. Math.}, 7:219--237, 2002.

\bibitem{Osin(2011_rankgradient)}
D.~Osin.
\newblock Rank gradient and torsion groups.
\newblock {\em Bull. Lond. Math. Soc.}, 43(1):10--16, 2011.

\bibitem{Pestov(2008)}
V.~G. Pestov.
\newblock Hyperlinear and sofic groups: a brief guide.
\newblock {\em Bull. Symbolic Logic}, 14(4):449--480, 2008.

\bibitem{Popa(2007)}
S.~Popa.
\newblock Deformation and rigidity for group actions and von {N}eumann
  algebras.
\newblock In {\em International Congress of Mathematicians. Vol. I}, pages
  445--477. Eur. Math. Soc., Z\"urich, 2007.

\bibitem{Raimbault(2012abelian)}
J.~Raimbault.
\newblock Exponential growth of torsion in abelian coverings.
\newblock {\em Algebr. Geom. Topol.}, 12(3):1331--1372, 2012.

\bibitem{Ranicki(1992)}
A.~A. Ranicki.
\newblock {\em Algebraic ${L}$-theory and topological manifolds}.
\newblock Cambridge University Press, Cambridge, 1992.

\bibitem{Ray-Singer(1971)}
D.~B. Ray and I.~M. Singer.
\newblock ${R}$-torsion and the {L}aplacian on {R}iemannian manifolds.
\newblock {\em Advances in Math.}, 7:145--210, 1971.

\bibitem{Reich(2006)}
H.~Reich.
\newblock {$L\sp 2$}-{B}etti numbers, isomorphism conjectures and
  noncommutative localization.
\newblock In {\em Non-commutative localization in algebra and topology}, volume
  330 of {\em London Math. Soc. Lecture Note Ser.}, pages 103--142. Cambridge
  Univ. Press, Cambridge, 2006.

\bibitem{Sanchez-Peralta(2023)}
P.~S{\'a}nchez-Peralta.
\newblock On vanishing criteria of {$L^2$}-{B}etti numbers of groups.
\newblock Preprint, arXiv:2307.07031 [math.GR], 2023.

\bibitem{Sauer(2009amen)}
R.~Sauer.
\newblock Amenable covers, volume and {$L^2$}-{B}etti numbers of aspherical
  manifolds.
\newblock {\em J. Reine Angew. Math.}, 636:47--92, 2009.

\bibitem{Sauer(2016)}
R.~Sauer.
\newblock Volume and homology growth of aspherical manifolds.
\newblock {\em Geom. Topol.}, 20(2):1035--1059, 2016.

\bibitem{Sauer-Thom(2010)}
R.~Sauer and A.~Thom.
\newblock A spectral sequence to compute {$L^2$}-{B}etti numbers of groups and
  groupoids.
\newblock {\em J. Lond. Math. Soc. (2)}, 81(3):747--773, 2010.

\bibitem{Schafer(1970)}
J.~A. Schafer.
\newblock Topological {P}ontrjagin classes.
\newblock {\em Comment. Math. Helv.}, 45:315--332, 1970.

\bibitem{Schick(2001b)}
T.~Schick.
\newblock ${L}\sp 2$-determinant class and approximation of ${L}\sp 2$-{B}etti
  numbers.
\newblock {\em Trans. Amer. Math. Soc.}, 353(8):3247--3265 (electronic), 2001.

\bibitem{Schlage-Puchta(2012)}
J.-C. Schlage-Puchta.
\newblock A {$p$}-group with positive rank gradient.
\newblock {\em J. Group Theory}, 15(2):261--270, 2012.

\bibitem{Schmidt(1995)}
K.~Schmidt.
\newblock {\em Dynamical systems of algebraic origin}, volume 128 of {\em
  Progress in Mathematics}.
\newblock Birkh\"auser Verlag, Basel, 1995.

\bibitem{Schmidt(2005)}
M.~Schmidt.
\newblock {\em {$L\sp 2$}-Betti Numbers of $\mathcal {R}$-Spaces and the
  Integral Foliated Simplicial Volume}.
\newblock Ph{D}~thesis, WWU~M\"unster, 2005.
\newblock Available online at
  http://nbn-resolving.de/urn:nbn:de:hbz:6-05699458563.

\bibitem{Scholze-Clausen(2019analyticgeometry)}
P.~Scholze and D.~Clausen.
\newblock Lectures on analytic geometry.
\newblock lecture notes URL:
  https://people.mpim-bonn.mpg.de/scholze/Analytic.pdf, 2019.

\bibitem{Scholze-Clausen(2019condensed)}
P.~Scholze and D.~Clausen.
\newblock Lectures on condensed mathematics.
\newblock lecture notes URL:
  https://people.mpim-bonn.mpg.de/scholze/Condensed.pdf, 2019.

\bibitem{Scholze-Clausen(2022condensedcomplex)}
P.~Scholze and D.~Clausen.
\newblock Condensed mathematics and complex geometry.
\newblock lecture notes URL:
  https://people.mpim-bonn.mpg.de/scholze/Complex.pdf, 2022.

\bibitem{Shorey-Tijdeman(1986)}
T.~N. Shorey and R.~Tijdeman.
\newblock {\em Exponential {D}iophantine equations}.
\newblock Cambridge University Press, Cambridge, 1986.

\bibitem{Uschold(2022)}
M.~Uschold.
\newblock Torsion homology growth and cheap rebuilding of inner-amenable
  groups.
\newblock Preprint, arXiv:2212.07916 [math.GR], 2022.

\bibitem{Wall(1967)}
C.~T.~C. Wall.
\newblock Poincar\'e complexes. {I}.
\newblock {\em Ann. of Math. (2)}, 86:213--245, 1967.

\end{thebibliography}

\end{document}